\documentclass[11pt,reqno]{amsart}
\usepackage{amssymb}
\usepackage{amsthm}
\usepackage{eucal}

\usepackage{a4wide}
\usepackage{color}
\usepackage{mathabx}

\newcommand{\R}{\mathbb R}

\newcommand{\C}{\mathbb C}

\newcommand{\p}{\partial}
\newcommand{\la}{\langle}
\newcommand{\ra}{\rangle}

\newcommand{\secao}[1]{\section{#1}\setcounter{equation}{0}}

\newtheorem{theorem}{Theorem}[section]
\newtheorem{proposition}[theorem]{Proposition}
\newtheorem{remark}[theorem]{Remark}
\newtheorem{lemma}[theorem]{Lemma}

\numberwithin{equation}{section}

\begin{document}
\title[Sharp well-posedness]{Sharp well-posedness for a coupled system of  mKdV  type equations}
\author{Xavier Carvajal}
\address{Instituto de Matem\'atica, Universidade Federal do Rio de Janeiro-UFRJ. Ilha do Fund\~{a}o,
21945-970. Rio de Janeiro-RJ, Brazil}
\email{carvajal@im.ufrj.br}
\author{Mahendra Panthee}
\address{Department of Mathematics, State University of Campinas, Brazil}
\email{mpanthee@ime.unicamp.br}
\thanks{M. Panthee was partially supported by CNPq (308131/2017-7) and FAPESP (2016/25864-6) Brazil.}


\begin{abstract}
We consider the initial value problem associated to a  system consisting     modified Korteweg-de Vries  type equations 
\begin{equation*}
\begin{cases}
\p_tv + \p_x^3v + \p_x(vw^2) =0,&v(x,0)=\phi(x),\\
\p_tw + \alpha\p_x^3w + \p_x(v^2w) =0,& w(x,0)=\psi(x),
\end{cases}
\end{equation*}
and prove  the local well-posedness results for   given data in low regularity Sobolev spaces
 $H^{s}(\R)\times H^{s}(\R)$, $s> -\frac12$, for $0<\alpha<1$. Our result covers the whole scaling sub-critical range of Sobolev regularity contrary to the case $\alpha =1$, where the local well-posedness holds only for $s\geq \frac14$.  We also prove that the local well-posedness result is sharp in two different ways, viz., for $s<-\frac12$ the key trilinear estimates used in the proof of the local well-posedness theorem fail to hold, and the flow-map that takes initial data to the solution fails to be $C^3$ at the origin. These results hold for $\alpha>1$ as well.
\end{abstract}

\maketitle

\secao{Introduction}
In this work we consider the initial value problem (IVP) associated to the  following  system of the  modified Korteweg-de Vries (mKdV) type equations
\begin{equation}\label{ivp-sy}
\begin{cases}
\p_tv + \p_x^3v + \p_x(vw^2) =0,&v(x,0)=\phi(x),\\
\p_tw + \alpha\p_x^3w + \p_x(v^2w) =0,& w(x,0)=\psi(x),
\end{cases}
\end{equation}
where   $(x, t) \in \R\times \R$; $v=v(x,t)$ and $w= w(x, t)$ are real-valued functions,  and $0<\alpha<1$ is a constant. 

For $\alpha =1$, the  system \eqref{ivp-sy} reduces to a special case of a broad class of nonlinear evolution equations considered by Ablowiz, Kaup, Newell and Segur \cite{AKNS} in the inverse scattering context. In this case, the well-posedness issues along with existence and stability of solitary waves for this system is widely studied in the literature.
Using the technique developed by Kenig, Ponce and Vega \cite{KPV3}, Montenegro \cite{Monte} proved that the IVP (\ref{ivp-sy}) with $\alpha =1$ is locally
well-posed for given data $(\phi, \psi)$ in $H^s(\R)\times H^s(\R)$, $s\geq \frac{1}{4}$. In this approach one uses the smoothing property of the linear group combined with the $L_x^pL_t^q$ Strichartz estimates and maximal function estimates. Tao \cite{Tao} showed that this local result can also be proved by using the Fourier transform restriction norm space $X_{s,b}$ (see definition \eqref{xsb}  below)  introduced by Bourgain \cite{B-93}. In this method the trilinear estimate
\begin{equation}
\label{Tri-ln}
\|\p_x(uvw)\|_{X_{s,b'}}\lesssim \|u\|_{X_{s,b}}\|v\|_{X_{s,b}}\|w\|_{X_{s,b}}
\end{equation}
that is valid for $s\geq \frac14$ plays a central role to apply contraction mapping principle. Author in \cite{Monte} also proved global well-posedness for given data in $H^s(\R)\times H^s(\R)$, $s\geq 1$, using the conservation laws
\begin{equation*}\label{con21}
I_1(v, w) := \int_{\R} (v^2 +w^2)\,dx
\end{equation*}
and
\begin{equation*}\label{con22}
I_2(v, w) := \int_{\R} (v_x^2 + w_x^2 - v^2w^2)\, dx,
\end{equation*}
satisfied by the flow of (\ref{ivp-sy}). This global result is further improved in \cite{CP} by proving for data with regularity $s>\frac14$. For existence and estability of solitary waves to the system  \eqref{ivp-sy} we refer to works in  \cite{AGM1} and \cite{Monte}. It is worth noting that the local well-posedness result for the system \eqref{ivp-sy} with $\alpha =1$ is sharp as it can be justified in two different way. First, the key trilinear estimate \eqref{Tri-ln} fails whenever $s<\frac14$, see \cite{KPV-b}. Second,  the associated IVP  for $s<\frac14$ is ill-posed in the sense that the mapping data-solution is not uniformly continuous, see \cite{KPV-01}. This notion of ill-posedness is a bit strong. For further works in this direction, we refer to \cite{CCT}.  

For  $0<\alpha<1$, very less is known regarding well-posedness issues for the IVP \eqref{ivp-sy}. In this work, we are interested to deal with these issues for given data in low regularity Sobolev spaces considering $0<\alpha<1$.  We note that, the approach of Kenig, Ponce, Vega \cite{KPV3} yields local well-posedness for $s\geq\frac14$ for $\alpha\ne 1$ too. However, if one uses the Fourier transform restriction norm space the situation is quite different. For motivation, we recall the work of Oh \cite{Oh09, Oh091} for the KdV and the Majda-Biello system introduced in  \cite{MB}
\begin{equation}\label{Majda-Biello}
\begin{cases}
\p_tv +\p_x^3v +\p_x(w^{2})=0,&v(x,0)=\phi(x),\\
\p_tw + \alpha \p_x^3w +\p_x (w v)=0,& w(x,0)=\psi(x),
\end{cases}
\end{equation}
 where $0<\alpha<1$. The author in \cite{Oh09} used the Fourier transform restriction norm method and proved that the IVP \eqref{Majda-Biello} is locally well-posed for data with regularity $s\geq 0$. He also showed that the well-posedness result  is sharp in the sense that if one demands $C^2$ regularity for the flow-map, the condition $s\geq 0$ is necessary. The main tool in the proof was the validity of the bilinear estimate for the interacting nonlinearity 
 \begin{equation}
\label{Bii-ln}
\|\p_x(vw)\|_{X_{s,b'}^{\alpha}}\lesssim \|v\|_{X_{s,b}}\|w\|_{X_{s,b}^{\alpha}}
\end{equation}
whenever $s\geq 0$ (for definition of $X_{s,b}$ and $X_{s,b}^{\alpha}$ spaces, see \eqref{xsb} and \eqref{xsb-a}  below.) Recall that the bilinear estimate \eqref{Bii-ln} in the case $\alpha =1$ holds for $s>-\frac34$ (see \cite{KPV-b}). Observe that, the regularity requirement for the validity of the bilinear estimates for the interacting nonlinearity in the case  $0<\alpha<1$ is higher than that in the case $\alpha = 1$. In the Fourier transform restrictionn norm method one needs to restrict the Fourier transform in the cubics $\tau = \xi^3$ and $\tau=\alpha\xi^3$. When $\alpha=1$, both the cubics are the same and one can obtain the similar  bilinear and trilinear estimates as in the individual KdV and mKdV equations. However, if $\alpha\ne 1$ the cubics are different and the frequency interactions behave differently. As observed  in \cite{Oh09} while studying a coupled system of the KdV equations and the Majda-Bielo system, no cancellation occurs and consequently needs higher regularity in the data to get required bilinear estimates.  Now the natural question is, how about the trilinear estimate if one considers $\alpha\ne 1$? 
 
 The main focus of this work is to answer the question posed above. In fact, considering   $0<\alpha< 1$ we prove that the trilinear estimate for the interacting nonlinear terms holds true whenever $s>- \frac12$. As a consequence we obtain  the local well-posedness result for the IVP \eqref{ivp-sy} for given data $(\phi,\psi)\in H^{s}(\R)\times H^{s}(\R)$,  $s>- \frac12$.  More precisely, we prove the following result.

\begin{theorem}\label{loc-sys}
Let $0<\alpha<1$ and $s> -\frac12$, then for any $(\phi,\psi)\in H^{s}(\R)\times H^{s}(\R)$,  there exist $\delta = \delta(\|(\phi,\psi)\|_{H^{s}\times H^{s}})$ (with $\delta(\rho)\to \infty$ as $\rho\to 0$) and a unique solution $(v,w)\in X^{1, \delta}_{s, b}\times X^{\alpha, \delta}_{s, b}$ to the IVP \eqref{ivp-sy} in the time interval $[0, \delta]$. Moreover, the solution satisfies the estimate
\begin{equation*}\label{est-1s}
\|(v, w)\|_{X^{1, \delta}_{s, b}\times X^{\alpha, \delta}_{s, b}}\lesssim \|(\phi,\psi)\|_{H^{s}\times H^{s}},
\end{equation*}
where the norm $\|\cdot\|_{X^{ 1,\delta}_{s, b}}$ and  $\|\cdot\|_{X^{\alpha, \delta}_{s, b}}$ are as defined in \eqref{xsb-rest}.
\end{theorem}

 The local well-posedness result obtained in  Theorem \ref{loc-sys} improves the result obtained by smoothing effect combined with the $L_x^pL_t^q$ Strichartz estimates  and maximal function estimates.  The main ingredients in the proof of Theorem \ref{loc-sys} are the new trilinear estimates (for exact statement see Proposition \ref{prop1} below)
\begin{equation}\label{tlint-m1}
\|(vw^2)_x\|_{X_{s,b'}^{1}}\lesssim  \|v\|_{X_{s,b}^{1}} \|w\|_{X^{\alpha}_{s,b}}^2
\end{equation}
and
\begin{equation}\label{tlint-m2}
 \|(v^2w)_x\|_{X^{\alpha}_{s,b'}}\lesssim  \|v\|_{X_{s,b}^{1}}^2 \|w\|_{X^{\alpha}_{s,b}},
\end{equation}
 which hold for $s>-\frac12$ when $0<\alpha<1$.
 
 It is quite surprising to note that, in contrast to the bilinear estimate, for the validity of the  above trilinear estimates with  $0<\alpha<1$  the regularity requirement is quite lower than the one required for case with $\alpha=1$. In fact, if one considers $\alpha =1$ the above trilinear estimates fail to hold whenever $s<\frac14$ (see \cite{KPV-b}). As can be seen in the proof of Lemma \ref{lema2} below,  when $0<\alpha<1$, no cancellation occurs in the resonance relation and this can be exploited in a positive way to lower the regularity requirement for the validity of the trilinear estimate. However, as discussed above, in the case of the corresponding bilinear estimates the regularity requirement is higher when two different cubics are considered, see \cite{Oh09}.
 
 We emphasize that the trilinear estimates for $0<\alpha<1$ are valid in the whole scaling sub-critical range, i.e., $s>-\frac12$. This is in quite contrast to the case $\alpha =1$ where it holds only for $s>\frac14$. At this point, we recall the work in \cite{NTT} where the authors produce a counter example to prove failure of bilinear estimate for $s=-\frac34$. So, it is natural to ask whether a similar example can be constructed for the trilinear estimates \eqref{tlint-m1} and \eqref{tlint-m2} for $s=-\frac12$. Also, we mention the work in \cite{CHT} where authors prove existence of global solution to the mKdV equation for given data with negative Sobolev regularity. Although, the mapping data-solution fails to be uniformly continuous for $s<\frac14$, they got such global result by  obtaining an appropriate global in time $H^s$- {\em a priori} estimate for $-\frac18<s<\frac14$. So, the another natural question is, whether global solution to the system \eqref{ivp-sy} can be found in line of \cite{CHT}. We are working on these questions and will be answered elsewhere.
 
 We also prove that the local well-posedness result obtained in Theorem \ref{loc-sys} is sharp in two different ways. First, we prove that  the crucial trilinear estimates  \eqref{tlint-m1}  and \eqref{tlint-m2}  fail to hold if $s<-\frac12$. Next, we show that the application that takes the initial data to the solution fails to be $C^3$ whenever $s<-\frac12$. More precisely, we prove the following results.
 
 \begin{proposition}\label{prop1ill}
Let  $0<\alpha<1$. The trilinear estimates \eqref{tlint-m1} and \eqref{tlint-m2} fail to hold for any $b \in \R$ whenever  $s<-1/2$.
\end{proposition}
 
 \begin{theorem}\label{mainTh-ill1.1}
Let $0<\alpha<1$. For any $s< -\frac12$ and for given  $(\phi, \psi)\in H^s(\R)\times H^s(\R)$, there exist no time $T =T(\|(\phi, \psi)\|_{H^s\times H^s})$ such that the application that takes initial data $(\phi, \psi)$ to the solution $(v, w) \in C([0,T];H^s) \times C([0,T];H^s)$ to the IVP \eqref{ivp-sy} is $C^3$ at the origin.
\end{theorem}

 This negative result makes sense, because if one uses contraction mapping principle to prove local well-posedness, the flow-map turns out to be smooth. 
 
 \begin{remark}
 All the results stated above hold for $\alpha>1$ as well. This can be justified by using symmetry of the system and scaling $\tilde{v}(x, t) =v(\alpha^{-\frac13}x, t)$ and $\tilde{w}(x, t) =w(\alpha^{-\frac13}x, t)$, so that $(\tilde{u}, \tilde{v})$ satisfy
 \begin{equation*}
\begin{cases}
\p_t\tilde{v}+\frac{1}{\alpha} \p_x^3\tilde{v} + \p_x(\tilde{v}\tilde{w}^2) =0,\\
\p_t\tilde{w} + \p_x^3\tilde{w} + \p_x(\tilde{v}^2\tilde{w}) =0,
\end{cases}
\end{equation*}
with $0<\frac{1}{\alpha} <1$.
 \end{remark}
 
 Before finalizing this section, we outline the structure of this paper. In Section \ref{Prel-e} we introduce function spaces and other notations used in this work and record some preliminary estimates. In Section \ref{local-r} we derive the main trilinear estimate and use it to prove the local well-posedness result stated in Theorem \ref{loc-sys}. Finally, in Section \ref{ill-posd-T}, we prove the failure of the trilinear estimates and the ill-posedness result stated respectively in Proposition \ref{prop1ill} and Theorem \ref{mainTh-ill1.1}.

\section{Function Spaces and Preliminary Estimates}\label{Prel-e}

In this section we introduce the function spaces and notations that are used throughout this paper. Also, we will record some preliminary estimates that are essential to prove the well-posedness result stated in Theorem \ref{loc-sys}. As described in the previous section we will  use the Fourier transform restriction norm space, the so called Bourgain's space.

For $f:\R\times [0, T] \to \R$ we define the mixed Lebesgue space 
 $L_x^pL_T^q$ with norm defined by
\begin{equation*}
\|f\|_{L_x^pL_T^q} = \left(\int_{\R}\left(\int_0^T |f(x, t)|^q\,dt
\right)^{p/q}\,dx\right)^{1/p},
\end{equation*}
with usual modifications when $p = \infty$. We replace $T$ by $t$ if $[0, T]$ is the whole real line $\R$.

We use $\widehat{f}(\xi)$ to denote  the Fourier transform of $f(x)$ defined by
$$
\widehat{f}(\xi) = c \int_{\R}e^{-ix\xi}f(x)dx$$
and
$\widetilde{f}(\xi,\tau)$ to denote  the Fourier transform of $f(x,t)$ defined by
$$
\widetilde{f}(\xi, \tau) = c \int_{\R^2}e^{-i(x\xi+t\tau)}f(x,t)dxdt.$$

We use $H^s$  to denote the $L^2$-based Sobolev space of order $s\in\R$ with norm
$$\|f\|_{H^s(\R)} = \|\langle \xi\rangle^s \widehat{f}(\xi)\|_{L^2_{\xi}},$$
where $\langle \xi\rangle = 1+|\xi|$.

Next, we introduce the  Fourier transform restriction norm spaces, more commonly known as Bourgain's space in literature.

For $s, b \in \R$,  we define the Fourier transform restriction norm spaces $X_{s,b}(\R\times\R)$ and $X^{\alpha}_{s,b}(\R\times\R)$ as completion of a space of Schwartz class functions with respective norms
\begin{equation}\label{xsb}
\|f\|_{ X_{s, b}} = \|(1+D_t)^b U(t)f\|_{L^{2}_{t}(H^{s}_{x})} = \|\langle\tau -\xi^3\rangle^b\langle \xi\rangle^s \widetilde{f}(\xi, \tau)\|_{L^2_{\xi,\tau}},
\end{equation}
and
\begin{equation}\label{xsb-a}
\|f\|_{ X^{\alpha}_{s, b}} = \|(1+D_t)^b U^{\alpha}(t)f\|_{L^{2}_{t}(H^{s}_{x})} = \|\langle\tau-\alpha\xi^3\rangle^b\langle \xi\rangle^s \widetilde{f}(\xi, \tau)\|_{L^2_{\xi,\tau}},
\end{equation}
 where $U(t) = e^{-t\partial^{3}_{x}}$ and $U^{\alpha}(t) = e^{-t\alpha\partial^{3}_{x}}$ are the unitary groups associated to the linear problems, and the operator $(1+D_t)^b$ are defined via the Fourier transform as  $[U^{\alpha}(t)f]\,^{\widehat{}}(\xi)=e^{it\alpha \xi^3} \widehat{f}(\xi)$ and $[(1+D_t)^bf]\,^{\widehat{}}(\tau)=\langle\tau\rangle^b\widehat{f}(\tau)$.

If $b> \frac12$, the Sobolev lemma implies that, $ X_{s, b} \subset C(\R ; H^s_x(\R)).$ For any interval $I$, we define the localized spaces $X_{s,b}^I:= X_{s,b}(\R\times I)$ with norm
\begin{equation}\label{xsb-rest}
\|f\|_{ X_{s, b}(\R\times I)} = \inf\big\{\|g\|_{X_{s, b}};\; g |_{\R\times I} = f\big\}.
\end{equation}
Sometimes we use the definition $X_{s,b}^{1,\delta}:=\|f\|_{ X_{s, b}(\R\times [0, \delta])}$ and similar for $X_{s,b}^{{\alpha},\delta}$.

 We use $c$ to denote various  constants whose exact values are immaterial and may
 vary from one line to the next. We use $A\lesssim B$ to denote an estimate
of the form $A\leq cB$ and $A\sim B$ if $A\leq cB$ and $B\leq cA$. Also, we
use the notation $a+$ to denote $a+\epsilon$ for $0< \epsilon \ll 1$.



In sequel we record some linear and nonlinear estimates satisfied by the solution in  $X_{s,b}$ and $X^{\alpha}_{s,b}$ spaces. 

We define a cut-off function $\psi_1 \in C^{\infty}(\R;\; \R^+)$ which is even, such that $0\leq \psi_1\leq 1$ and
\begin{equation*}\label{cut-1}
\psi_1(t) = \begin{cases} 1, \quad |t|\leq 1,\\
                          0, \quad |t|\geq 2.
            \end{cases}
\end{equation*}
We also define $\psi_T(t) = \psi_1(t/T)$, for $0< T\leq 1$.

In what follows we list some estimates that are crucial in the proof of the local well-posedness result.
\begin{lemma}\label{lemma1}
For any $s, b \in \R$, we have
\begin{equation}\label{lin.1}
\|\psi_1U(t)\phi\|_{X_{s,b}}\leq C \|\phi\|_{H^s}, \qquad \|\psi_1U^{\alpha}(t)\phi\|_{X^{\alpha}_{s,b}}\leq C \|\phi\|_{H^s}.
\end{equation}
Further, if  $-\frac12<b'\leq 0\leq b<b'+1$ and $0\leq \delta\leq 1$, then
\begin{equation}\label{nlin.1}
\|\psi_{\delta}\int_0^tU(t-t')f(u(t'))dt'\|_{X_{s,b}}\lesssim \delta^{1-b+b'}\|f(u)\|_{X_{s, b'}}
\end{equation}
and
\begin{equation}\label{nlin.11}
\|\psi_{\delta}\int_0^tU^{\alpha}(t-t')f(u(t'))dt'\|_{X^{\alpha}_{s,b}}\lesssim \delta^{1-b+b'}\|f(u)\|_{X^{\alpha}_{s, b'}}.
\end{equation}
\end{lemma}
\begin{proof}
For the proof of this lemma we refer to \cite{GTV}.
\end{proof}

\begin{remark} In the proof of the local well-posedness results, we will take $b'=-\frac12+2\epsilon$ and $b=\frac12+\epsilon$ so that $1-b+b'$ is strictly positive.
\end{remark}

Now, we state the trilinear estimate which is central in our argument. 

\begin{proposition}\label{prop1}
Let $0<\alpha<1$,  $b>\frac12$, and $b'$ be as in Lemma \ref{lemma1}. Then  the following trilinear estimates 
\begin{equation}\label{tlint}
\|(vw^2)_x\|_{X_{s,b'}}\lesssim  \|v\|_{X_{s,b}} \|w\|_{X^{\alpha}_{s,b}}^2,
\end{equation}
and
\begin{equation}\label{tlin-2t}
\|(v^2w)_x\|_{X^{\alpha}_{s,b'}}\lesssim  \|v\|_{X_{s,b}}^2\|w\|_{X^{\alpha}_{s,b}},
\end{equation}
hold for any  $s>- \frac12$.
\end{proposition}

Proof of this proposition will be provided below in a separate section.

\secao{Proofs of the key trilinear estimates and local well-posedness result}\label{local-r}
In this section we prove the crucial tri-linear estimates stated in Proposition \ref{prop1} and the local well-posedness result stated in Theorem \ref{loc-sys}.
\subsection{Proof of the trilinear estimate}

  We start by  defining $n$-multiplier and $n$-linear functional, see \cite{Tao}.
Let $n\geq 2$ be an even integer. An $n$-multiplier $M_n(\xi_1, \dots, \xi_n)$ is a function defined on the hyper-plane $\Gamma_n:= \{(\xi_1, \dots, \xi_n);\;\xi_1+\dots +\xi_n =0\}$ with Dirac delta $\delta(\xi_1+\cdots +\xi_n)$ as a measure.

We define an $[n; \R^{n+1}]$-multiplier to be any function $m:\Gamma_n(\R^{n+1})\to \C$, and its norm $\|m\|_{[n; \R^{n+1}]}$ to be the best constant such that
\begin{equation}
\label{m-norm}
\Big|\int_{\Gamma_n}m(\xi)\Pi_{j=1}^nf_j(\xi)\Big|\leq \|m\|_{[n; \R^{n+1}]}\Pi_{j=1}^n\|f_j\|_{L^2(\R^{n+1})}
\end{equation}
holds true for all the test functions in $\R^{n+1}$.

If $M_n$ is an $n$-multiplier and $f_1, \dots, f_n$ are functions on $\R$, we define an $n$-linear functional, as
\begin{equation}\label{n-linear}
\Lambda_n(M_n;\; f_1, \dots, f_n):= \int_{\Gamma_n}M_n(\xi_1, \dots, \xi_n)\prod_{j=1}^{n}\hat{f_j}(\xi_j).
\end{equation}
We write $\Lambda_n(M_n):=\Lambda_n(M_n;\; f,f,\dots,f)$ in the case when $\Lambda_n$ is applied to the $n$ copies of the same function $f$.

The following results will  be useful in our argument. 

\begin{lemma}\label{lemxm1}
\begin{itemize}
\item[(i)] If $a,b>0$ and $a+b>1$, we have
\begin{equation}\label{lfc1-1}
\int_{\R} \dfrac{dx}{\la x -\alpha \ra^{a}\la x -\beta \ra^{b}} \lesssim \dfrac{1}{\la \alpha -\beta \ra^{c}},  \quad c=\min\{a,b, a+b-1\}.
\end{equation}

\item[(ii)]
Let $a, \eta \in \R$, $a, \eta  \neq 0$, $b>1$, then
\begin{equation}\label{xlfc1-2}
\int_{\R} \dfrac{dx}{\la a(x^2-\eta^2) \ra^{b}}  \lesssim \dfrac{1}{|a\eta|},
\end{equation}

\item[(iii)]
Let $a, \eta \in \R$, $a, \eta  \neq 0$, $b>1$, then
\begin{equation}\label{x1lfc1-2}
\int_{\R} \dfrac{|x\pm \eta|\,dx}{\la a(x^2\pm\eta^2) \ra^{b}} \lesssim \dfrac{1}{|a|},
\end{equation}

\item[(iv)] For $l>1/3$,
\begin{equation}\label{x2lfc1-2}
\int_{\R} \dfrac{dx}{\la x^3+a_2x^2 +a_1 x+a_0 \ra^{l}} \lesssim 1.
\end{equation}

\end{itemize}
\end{lemma}
\begin{proof}
Proof of \eqref{lfc1-1}  can be found in \cite{TAKA3},  \eqref{xlfc1-2} and \eqref{x1lfc1-2} in \cite{XC-04} and \eqref{x2lfc1-2}  in  \cite{BOP}.
\end{proof}

Now we prove some bilinear estimates that play central role in the proof of the trilinear estimates.
\begin{lemma}\label{lema2}
 Let $s>- \frac12$, $0<\epsilon<\frac{2s+1}{15}$  and $0<\alpha<1$, then the following bilinear estimates
\begin{equation}\label{bil-1.1}
\|uv\|_{L^2(\R^2)}\lesssim \|u\|_{X_{-\frac12, \frac12-2\epsilon}}\|v\|_{X^{\alpha}_{s, \frac12+\epsilon}},
\end{equation}
and
\begin{equation}\label{bil-2.1}
\|uv\|_{L^2(\R^2)}\lesssim \|u\|_{X^{\alpha}_{-\frac12, \frac12-2\epsilon}}\|v\|_{X_{s, \frac12+\epsilon}}
\end{equation}
hold true.
\end{lemma}

\begin{proof}
We provide proofs for \eqref{bil-1.1}, the proof of the bilinear estimate \eqref{bil-2.1} is similar.  Using Plancherel's identity, 
the estimate \eqref{bil-1.1} is equivalent to showing that
\begin{equation}\label{bil-3xc}
\| \mathcal{B}_s(f,g) \|_{L^2_\xi L^2_\tau} \le C \|f\|_{L^2(\R^2)} \|g\|_{L^2(\R^2)},
\end{equation}
where 
\begin{equation}\label{bil-f1}
\mathcal{B}_s(f,g)= \int_{\R^2} \dfrac{\la \xi_2 \ra^{1/2}\widetilde{f}(\xi_2, \tau_2)\widetilde{g}(\xi_1, \tau_1) }{\langle \xi_1 \rangle^s\la \tau_1-\alpha \xi_1^3 \ra^{\frac12+\epsilon} \la \tau_2-\xi_2^3 \ra^{\frac12-2\epsilon} } d\xi_1 d\tau_1, 
\end{equation}
with   $\widetilde{f}(\xi,\tau)=\la\xi\ra^{-\frac12}\la\tau-\xi^3\ra^{\frac12-2\epsilon} \widetilde{u}(\xi,\tau)$,  $\widetilde{g}(\xi,\tau)=\la\xi\ra^s\la\tau-\alpha\xi^3\ra^{\frac12+\epsilon} \widetilde{v}(\xi,\tau)$, $\xi_2=\xi-\xi_1$ and $\tau_2=\tau-\tau_1$. Using Cauchy-Schwarz  inequality we obtain \eqref{bil-3xc}  if
\begin{equation}\label{definL}
\mathcal{L}_1:=\sup_{\xi,\tau}\int_{\R^2} \dfrac{\la \xi_2 \ra}{\langle \xi_1 \rangle^{2s}\la \tau_1-\alpha \xi_1^3 \ra^{1+2\epsilon} \la \tau_2- \xi_2^3 \ra^{1-4\epsilon} } d\xi_1 d\tau_1 \lesssim 1.
\end{equation}

Applying the estimate \eqref{lfc1-1} for $\tau_1$ integral,  we obtain from \eqref{definL} that
\begin{equation}\label{definL1}
\mathcal{L}_1 \lesssim\sup_{\xi,\tau} \int_{\R} \dfrac{\la \xi_2 \ra}{\langle \xi_1 \rangle^{2s}\la \tau- \xi_2^3-\alpha \xi_1^3 \ra^{1-4\epsilon}  } d\xi_1:=\sup_{\xi,\tau} L_1,
\end{equation}
for $0<\epsilon<\frac12$.

Now, we move to estimate the integral in \eqref{definL1} dividing in four different cases:
\begin{equation*}
|\xi|\leq 1, \qquad 1<|\xi|\leq c_1|\xi_1|, \qquad  c_1|\xi_1|<|\xi|<\frac1{c_2}|\xi_1|, \qquad \frac1{c_2}|\xi_1|\leq|\xi|,
\end{equation*}
where $c_1=-1 +\dfrac{\sqrt{3+(3- \lambda)(1-\alpha)}}{\sqrt3}$ with $0<\lambda<3$ and $c_2=\dfrac{c_1}{1-\alpha}$.\\

\noindent
{\bf Case a) $\boxed{|\xi| \le 1}$. } Considering $c_0:= \dfrac{9}{1-\alpha}$, we further divide in two sub-cases: $$ |\xi_1| \ge c_0, \qquad {\mathrm{and}}\qquad  |\xi_1| < c_0.$$

\noindent
{\bf Subcase a1) $\boxed{|\xi_1| \ge c_0}.$}
 In this case  we have
\begin{equation}\label{finitexi}
 \la \xi_2 \ra \langle \xi_1 \rangle \lesssim \langle \xi_1 \rangle^2 \lesssim 1+ \xi_1^2.
\end{equation}
 For fixed $\xi$ and $\tau$, we define
\begin{equation}\label{H-1}
H(\xi_1)=\tau- \xi_2^3-\alpha \xi_1^3.
\end{equation}

We have 
\begin{equation}\label{1eqx1}
\begin{split}
H'(\xi_1)=&3(1-\alpha)\xi_1^2+3\xi^2-6\xi \xi_1
\geq 3(1-\alpha)\xi_1^2-3-6| \xi_1|
\geq 2(1-\alpha)\xi_1^2.
\end{split}
\end{equation}

The function $H'$ has two roots namely $$r_1=\frac{\xi+|\xi| \sqrt{\alpha} }{1-\alpha},\quad r_2=\frac{\xi-|\xi| \sqrt{\alpha} }{1-\alpha}.$$ 
Thus the function $y=H(\xi_1)$ is monotone on each of intervals: 
\begin{equation}\label{intervals}
(-\infty,  r_2), \quad [r_2 , r_1), \quad \textrm{and}  \quad (r_1, \infty).  
\end{equation}
Thus in this case, using that $2s+1>0$, we get 
\begin{equation}\label{1eqx2}
\begin{split}
L_1=& \int_{|\xi_1|\ge c_0} \dfrac{\la \xi_2 \ra \la \xi_1 \ra}{\langle \xi_1 \rangle^{2s+1}\la \tau- \xi_2^3-\alpha \xi_1^3 \ra^{1-4\epsilon}  } d\xi_1\\
 \lesssim & \int_{|\xi_1|\ge c_0} \dfrac{ 1+ |\xi_1|^2 }{\langle \xi_1 \rangle^{2s+1}\la \tau- \xi_2^3-\alpha \xi_1^3 \ra^{1-4\epsilon}  } d\xi_11\\
 \lesssim & \int_{|\xi_1|\ge c_0} \dfrac{ 1 }{\la \tau- \xi_2^3-\alpha \xi_1^3 \ra^{1-4\epsilon}  } d\xi_1+ \int_{|\xi_1|\ge c_0} \dfrac{ |\xi_1|^2 }{\langle \xi_1 \rangle^{2s+1}\la \tau- \xi_2^3-\alpha \xi_1^3 \ra^{1-4\epsilon}  } d\xi_1\\
 \lesssim & J_1+J_2.
\end{split}
\end{equation}

Using  \eqref{x2lfc1-2} we have  $J_1 \lesssim 1$ provided $0<\epsilon<\frac16$. In what follows, we estimate $J_2$ considering two different cases.

\noindent
{\bf Case a1.1)  $\boxed{\la H(\xi_1)\ra  \lesssim |\xi_1|^3}$.}  In this case, considering  $2s+1 > 0$, we get
\begin{equation}\label{1eqx3}
\begin{split}
J_2
 \lesssim & \int_{|\xi_1|\ge c_0} \dfrac{ H'(\xi_1)  }{| \xi_1 |^{2s+1}\la H(\xi_1) \ra^{1-4\epsilon}  } d\xi_1\\
\lesssim &  \int_{|\xi_1|\ge c_0} \dfrac{ H'(\xi_1) \la H(\xi_1) \ra^{5\epsilon} }{| \xi_1 |^{2s+1}\la H(\xi_1) \ra^{1+\epsilon}  } d\xi_1\\
\lesssim &  \int_{|\xi_1|\ge c_0} \dfrac{ H'(\xi_1) |\xi_1|^{15\epsilon}  }{| \xi_1 |^{2s+1}\la H(\xi_1) \ra^{1+\epsilon}  } d\xi_1.
\end{split}
\end{equation}
Making a  change  of variables of $x=H (\xi_1)$ on each interval of monotonicity of $H(\xi_1)$, for  $0<\epsilon <\frac{2s+1}{15}$, we obtain from \eqref{1eqx3} that
\begin{equation*}
\begin{split}
J_2 \lesssim \int_{\R} \dfrac{ dx}{ \la x\ra^{1+\epsilon}  } \lesssim 1.
\end{split}
\end{equation*}

\noindent
{\bf Case a1.2)    $\boxed{\la H(\xi_1)\ra \gtrsim |\xi_1|^3}$.} In this case, one can obtain
\begin{equation*}
\begin{split}
J_2 \lesssim& \int_{|\xi_1|\ge c_0} \dfrac{| \xi_1|^{2}  }{| \xi_1|^{2s+1} | \xi_1|^{3-12\epsilon}  } d\xi_1
\\
\lesssim & \int_{|\xi_1|\ge c_0} \dfrac{1}{ | \xi_1|^{2s+2-12\epsilon} } d\xi_1
\lesssim  1,
\end{split}
\end{equation*}
where the last inequality we used that $2s+2-12 \epsilon >1$, which  holds for $0<\epsilon <\frac{2s+1}{12}$. This completes the proof of the {\bf Subcase  a1)}.\\

\noindent
{\bf Subcase a2)  $\boxed{|\xi_1| \le c_0}$.}
 Using that $2s+1>0$ and \eqref{x2lfc1-2}, in this case  we get 
\begin{equation}\label{1eqx2-m2}
\begin{split}
L_1=& \int_{|\xi_1|\le c_0} \dfrac{\la \xi_2 \ra \la \xi_1 \ra}{\langle \xi_1 \rangle^{2s+1}\la \tau- \xi_2^3-\alpha \xi_1^3 \ra^{1-4\epsilon}  } d\xi_1\\
 \lesssim & \int_{\R} \dfrac{1 }{\la \tau- \xi_2^3-\alpha \xi_1^3 \ra^{1-4\epsilon}  } d\xi_1
 \lesssim 1,
\end{split}
\end{equation}
provided $0<\epsilon<\frac16$.\\

\noindent
{\bf Case b)   $\boxed{1<|\xi| \leq c_1 |\xi_1|}$.} 
In this case, we have 
\begin{equation}\label{21eqx1-m}
\begin{split}
H'(\xi_1)=&3(1-\alpha)\xi_1^2+3\xi^2-6\xi \xi_1
\geq 3(1-\alpha)\xi_1^2-3\xi^2-6|\xi \xi_1|
\geq \lambda(1-\alpha)\xi_1^2.
\end{split}
\end{equation}
In the last inequality we considered the value of $c_1$ as the positive root of $3c_1^2+6c_1=(3- \lambda)(1-\alpha)$, i.e,   $c_1=-1 +\dfrac{\sqrt{3+(3- \lambda)(1-\alpha)}}{\sqrt3}$, $0<\lambda<3$.  Again, we divide in two sub-cases:  $|\xi_1| \ge 1$ and  $|\xi_1| \leq 1$.
\\

\noindent
{\bf Subcase b1) $\boxed{|\xi_1| \ge 1}$.}
 In this case  we get
\begin{equation}\label{2finitexi-m}
 \la \xi_2 \ra \langle \xi_1 \rangle \lesssim \langle \xi_1 \rangle^2 \lesssim  \xi_1^2.
\end{equation}

Using the inequalities  \eqref{21eqx1-m}, \eqref{2finitexi-m} and similar arguments as in the above {\bf Subcase  a1)} we conclude that $L_1 \lesssim 1$.
\\

\noindent
{\bf Subcase b2) $\boxed{|\xi_1| \le 1}$.}
In this case also $|\xi| \lesssim 1$, and similarly as in the {\bf Subcase  a2)} we have $L_1 \lesssim 1$.\\

\noindent
{\bf Case c)    $\boxed{\frac1{c_2}|\xi_1| \leq |\xi|$,  $c_2=\dfrac{c_1}{1-\alpha}}$.} In this case, 
we have 
\begin{equation}\label{c2finitexi}
  \la \xi_2 \ra \langle \xi_1 \rangle \lesssim  \la \xi \ra \langle \xi_1 \rangle\quad \textrm{and}\quad \la \xi_2 \ra \langle \xi \rangle \lesssim  \la \xi \ra^2 .
\end{equation}
Also
\begin{equation}\label{21eqx1}
\begin{split}
H'(\xi_1)=3(1-\alpha)\xi_1^2+3\xi^2-6\xi \xi_1
\geq 3\xi^2-3(1-\alpha)\xi_1^2-6|\xi \xi_1| \geq \lambda \xi^2,
\end{split}
\end{equation}
where we considered  $c_2$ as the positive root of $3(1-\alpha)c_2^2+6c_2=(3- \lambda)$ in the last inequality. In view of {\bf Case a)}, we can consider  $|\xi| \ge 1$. Here also, considering $b=\dfrac1{1-4\epsilon},\;\; a+b=3$, we divide in two sub-cases.\\

\noindent
{\bf Subcase c1)  $\boxed{\la H(\xi_1)\ra  \lesssim \la \xi_1 \ra^a \la \xi \ra^b}$.}  In this case, considering  $2s+1 > 0$, we get
\begin{equation}\label{c1eqx3}
\begin{split}
L_1=& \int_{\R} \dfrac{\la \xi_2 \ra }{\langle \xi_1 \rangle^{2s}\la \tau- \xi_2^3-\alpha \xi_1^3 \ra^{1-4\epsilon}  } d\xi_1\\
 \lesssim & \int_{\R} \dfrac{  H'(\xi_1)  }{\la \xi \ra \langle \xi_1 \rangle^{2s}\la H(\xi_1) \ra^{1-4\epsilon}  } d\xi_1\\
\lesssim &  \int_{\R} \dfrac{ H'(\xi_1) \la H(\xi_1) \ra^{5\epsilon} }{\la \xi_1 \ra^{2s} \la \xi \ra \la H(\xi_1) \ra^{1+\epsilon}  } d\xi_1\\
\lesssim &  \int_{\R} \dfrac{ H'(\xi_1)  }{\la \xi_1 \ra^{2s-5 a\epsilon} \la \xi \ra^{1-5b \epsilon}\la H(\xi_1) \ra^{1+\epsilon}  } d\xi_1\\
\lesssim &  \int_{\R} \dfrac{ H'(\xi_1)  }{\la  \xi_1 \ra^{2s+1-5 (a+b)\epsilon}\la H(\xi_1) \ra^{1+\epsilon}  } d\xi_1.
\end{split}
\end{equation}

Recall  that $a+b=3$,  taking $0<\epsilon < \dfrac{2s+1}{15}$ and  making a  change  of variables of $x=H (\xi_1)$ on each interval of monotonicity of $H(\xi_1)$, we obtain from \eqref{c1eqx3} that
\begin{equation*}
\begin{split}
L_1 \lesssim \int_{\R} \dfrac{ dx}{ \la x\ra^{1+\epsilon}  } \lesssim 1.
\end{split}
\end{equation*}

\noindent
{\bf Subcase c2)    $\boxed{\la H(\xi_1)\ra \gtrsim \la \xi_1 \ra^a \la \xi \ra^b}$.} In this case, one can obtain
\begin{equation*}
\begin{split}
L_1 \lesssim& \int_{\R} \dfrac{\la  \xi_1\ra \la \xi \ra  }{\la  \xi_1\ra^{2s+1} \la \xi_1\ra^{a(1-4 \epsilon)} \la \xi \ra^{b(1-4 \epsilon)}  } d\xi_1\\
= & \int_{\R} \dfrac{\la  \xi_1\ra \la \xi \ra  }{\la \xi_1\ra^{2s+1} \la \xi_1\ra^{2-12\epsilon}\la \xi\ra  } d\xi_1\\
= & \int_{\R} \dfrac{1}{ \la \xi_1\ra^{2s+2-12\epsilon} } d\xi_1
\lesssim  1,
\end{split}
\end{equation*}
where in the last inequality we used that $2s+2-12 \epsilon >1$ which  holds for $0<\epsilon <\frac{2s+1}{12}$. This completes the proof of the {\bf  Case c)}.
\\

\noindent
{\bf Case d)  $\boxed{c_1|\xi_1| <|\xi| < \dfrac{1}{c_2} |\xi_1|}$:} 
In view of {\bf  Case a)} we can suppose   $|\xi| >1 $, and consequently   $|\xi_1| \sim |\xi|>1$. Let $\mathcal{R}= \{\xi_1  ;   \,\, c_2|\xi| <|\xi_1| < \dfrac{1}{c_1} |\xi| \}$. We will consider two different cases:  $|H(\xi_1)| \gtrsim |\xi|^3$ and  $|H(\xi_1)| \lesssim |\xi|^3$.\\

\noindent
{\bf Subcase d1)    $\boxed{|H(\xi_1)| \gtrsim |\xi|^3}$.} Since $|\xi_1| \sim |\xi|>1$ one has $\la \xi_2 \ra \lesssim \la\xi_1\ra$ and  therefore
\begin{equation*}
\begin{split}
L_1 \lesssim& \int_{\mathcal{R}} \dfrac{\la \xi_2 \ra  }{\langle \xi_1 \rangle^{2s}\la  \xi \ra^{3-12\epsilon}  } d\xi_1
\\
\lesssim & \int_{\R} \dfrac{1}{\langle \xi_1 \rangle^{2s+2-12\epsilon} } d\xi_1 
\lesssim  1,
\end{split}
\end{equation*}
where the last inequality holds for $0<\epsilon <\frac{2s+1}{12}$. \\

\noindent
{\bf Subcase d2)  $\boxed{|H(\xi_1)| \lesssim |\xi|^3}$.}  In this case, we get 
\begin{equation}\label{La.1}
\begin{split}
L_1 =& \int_{\mathcal{R}} \dfrac{\la \xi_2 \ra \la \xi_1 \ra  \la H (\xi_1)\ra^{5\epsilon}}{\langle \xi_1 \rangle^{2s+1}\la H (\xi_1)\ra^{1+\epsilon}  } d\xi_1
\\
\lesssim & \int_{\mathcal{R}} \dfrac{\la \xi \ra^2 \la \xi\ra^{15\epsilon}}{\langle \xi \rangle^{2s+1}\la H (\xi_1)\ra^{1+\epsilon}  } d\xi_1 
\\
\lesssim & \dfrac{\la \xi \ra^2}{\la \xi\ra^{2s+1-15\epsilon} }\int_{\mathcal{R}} \dfrac{ 1}{ \la H (\xi_1)\ra^{1+\epsilon}  } d\xi_1.
\end{split}
\end{equation}

Note that
\begin{equation}\label{xH-1}
H(\xi_1)=\tau- \xi_2^3-\alpha \xi_1^3= \tau -\xi^3+ (1-\alpha) \xi_1^3-3\xi \xi_1^2 +3 \xi^2 \xi_1.
\end{equation}

Making a change of variables $\xi_1= \xi x$, we obtain  that
\begin{equation}\label{xH-2}
\begin{split}
\mathcal{X}:= \int_{\mathcal{{R}}} \dfrac{ 1}{ \la H (\xi_1)\ra^{1+\epsilon}  } d\xi_1=&\int_{\mathcal{\tilde{R}}} \dfrac{ |\xi|}{ \la \tau -\xi^3+ \xi^3 [(1-\alpha) x^3-3 x^2 +3 x ]\ra^{1+\epsilon}  } dx\\
=& \int_{\mathcal{\tilde{R}}} \dfrac{ |\xi|}{ \la \xi^3 \kappa^{-1}[c_0 +\kappa^3+ (x-\kappa)^3-3 x\kappa^2 \alpha ]\ra^{1+\epsilon}  } dx,
\end{split}
\end{equation}
 where $c_0=\dfrac{\tau -\xi^3}{\xi^3}\kappa$ and $\kappa= \dfrac{1}{1-\alpha}$, $\mathcal{\tilde{R}}=\{x  ;   \,\, c_2<|x| < \dfrac{1}{c_1}  \}$. 

Again, we make a change of variables $x=\kappa(\eta+1)$,  next the change of variables $\eta =\omega+ \dfrac{\alpha}{\omega}$ and finally the change of variables $\omega^3= t$, to arrive at
\begin{equation*}
\begin{split}
\mathcal{X} \sim   |\xi|\int_{\mathcal{A}_{\alpha}} \dfrac{ |t^{2/3} -\alpha|}{ |t|^{1/3-\epsilon} ( |t|+ |\xi|^3 |t^2+\kappa_1 t+ \alpha^3 |)^{1+\epsilon}  } dt,
\end{split}
\end{equation*}
where $\kappa_1=c_0(1-\alpha)^3+1-3\alpha$  and $$\mathcal{A}_{\alpha}=\left\{t : \,\, \left| |t^{1/3}+ \dfrac{\alpha}{t^{1/3}}+1| -\sqrt{2-\alpha} \right|<1\right\}.$$

It is not difficult to see that 
\begin{equation}\label{0xH-4}
0<s_1(\alpha) < |t| < s_2(\alpha).
\end{equation}
 Thus $|t| \sim_{\alpha} 1$, and hence
\begin{equation}\label{xH-3}
\begin{split}
\mathcal{X} \sim   |\xi|\int_{\mathcal{A}_{\alpha}} \dfrac{ |t^{2/3} -\alpha|}{  ( 1+ |\xi|^3 |t^2+\kappa_1 t+ \alpha^3 |)^{1+\epsilon}  } dt.
\end{split}
\end{equation}
Let the discriminant be $\Delta=\kappa_1^2-4\alpha^3$. We consider three cases:  $\Delta=0$,  $\Delta>0$ and  $\Delta<0$.
\\

\noindent
{\bf Subcase d2.1)  $\boxed{\Delta=0}$:}
In this case $\kappa_1=\pm2 \alpha^{3/2}$ and 
\begin{equation}\label{0xH-5}
t^2+\kappa_1 t+ \alpha^3=(t+\dfrac{\kappa_1}2)^2=(t\pm \alpha^{3/2})^2.
\end{equation}

We will consider the positive sign in \eqref{0xH-5}, the estimate with negative sign is similar. From \eqref{xH-3}, \eqref{0xH-4}  and \eqref{0xH-5} we obtain
\begin{equation}\label{xH-6}
\begin{split}
\mathcal{X} \lesssim&   |\xi|\int_{\mathcal{A}_{\alpha}} \dfrac{ |t^{2/3} -\alpha|}{  ( 1+ |\xi|^3 |t^2+\kappa_1 t+ \alpha^3|)^{1+\epsilon}  } dt\\
\lesssim &   |\xi|\int_{\mathcal{A}_{\alpha}} \dfrac{|(t^{1/3} -\sqrt{\alpha})(t^{1/3} +\sqrt{\alpha})|}{  ( 1+ |\xi|^3 |t+ (\sqrt{\alpha})^3|^2 )^{1+\epsilon}  } dt\\
\lesssim &   |\xi|\int_{\mathcal{A}_{\alpha}} \dfrac{|(t^{1/3} +\sqrt{\alpha})|}{  ( 1+ |\xi|^3 |t^{1/3}+ \sqrt{\alpha}|^2|t^{2/3}-t^{1/3}\sqrt{\alpha} +\alpha|^2 )^{1+\epsilon}  } dt\\
\lesssim &   |\xi|\int_{\mathcal{A}_{\alpha}} \dfrac{|(t^{1/3} +\sqrt{\alpha})|}{  ( 1+ |\xi|^3 |t^{1/3}+ \sqrt{\alpha}|^2 )^{1+\epsilon}  } dt,
\end{split}
\end{equation}
where we used that 
\begin{equation}\label{cxH-6}
x^2\pm xc+c^2 \geq \dfrac{3c^2}{4}, \quad \forall x, c \in \R.
\end{equation}
 Finally we make a change of variable $u=|\xi|^{3/2} (t^{1/3} +\sqrt{\alpha})$ in the last integral. Considering $t\in \mathcal{A}_{\alpha}$, we have $\mathcal{X} \lesssim \dfrac1{|\xi|^2}$.
Therefore in this case, from \eqref{La.1}, we have
\begin{equation}\label{xLa.1}
\begin{split}
L_1 
\lesssim  \dfrac{\la \xi \ra^2}{\la \xi\ra^{2s+1-15\epsilon} }\mathcal{X}
\lesssim  \dfrac{\la \xi \ra^2}{\la \xi\ra^{2s+1-15\epsilon} } \dfrac1{|\xi|^2}
\lesssim  1,
\end{split}
\end{equation}
in the last inequality   $0<\epsilon <\frac{2s+1}{15}$ was used.

For the sake of simplicity, without loss of generality, from here onwards we consider $\alpha=\frac12$ (the general case $\alpha \in (0,1)$ follows analogously). Observe that 
\begin{equation}\label{xH-4}
\begin{split}
\mathcal{A}_{1/2}=&\left\{t : \dfrac{1}{4}(-4-\sqrt{6}) -\dfrac{1}{4}\sqrt{14+8\sqrt{6}}< t^{1/3}< \dfrac{1}{4}(-4-\sqrt{6}) +\dfrac{1}{4}\sqrt{14+8\sqrt{6}}\right\}\\
=&(l_1, \, l_2),
\end{split}
\end{equation}
where $l_1=-28.6526\dots < l_2=-0.004356\dots <0$.
\\

\noindent
{\bf Subcase d2.2)   $\boxed{\Delta>0}$:}
In this case $|\kappa_1|>2 \alpha^{3/2}$ and 
\begin{equation}\label{00xH-5}
\begin{split}
t^2+\kappa_1 t+ \alpha^3=(t+\dfrac{\kappa_1}2)^2+\alpha^3 -\dfrac{\kappa_1^2}{4}
= (t+\beta)^2-\gamma^2,
\end{split}
\end{equation}
where $\beta= \dfrac{\kappa_1}2$ and $\gamma^2=\beta^2-\alpha^3$ consequently $\beta=\pm(\gamma^2+\alpha^3)^{1/2}$. We consider  $\beta=(\gamma^2+\alpha^3)^{1/2}$ similar argument works if $\beta=-(\gamma^2+\alpha^3)^{1/2}$. Considering a change of variable $\tau=t+\beta$ in \eqref{xH-3}, we get
\begin{equation}\label{xH-9}
\begin{split}
\mathcal{X} \lesssim &   |\xi|\int_{l_1+\beta}^{l_2+\beta} \dfrac{ |(\tau -\beta)^{2/3} -\alpha|}{ ( 1+ |\xi|^3 |\tau^2-\gamma^2|)^{1+\epsilon}  } d\tau\\
\lesssim &   |\xi|\int_{l_1+\beta}^{l_2+\beta}\dfrac{ |(\tau -(\gamma^2+\alpha^3)^{1/2})^{2/3} -\alpha|}{ ( 1+ |\xi|^3 |\tau^2-\gamma^2|)^{1+\epsilon}  } d\tau.
\end{split}
\end{equation}
As $l_1 \leq \tau - \beta =\tau -(\gamma^2+\alpha^3)^{1/2} \leq  l_2$, it is not difficult to see
\begin{equation}\label{xH-10}
\begin{split}
\dfrac{ |(\tau -(\gamma^2+\alpha^3)^{1/2})^{2/3} -\!\alpha|}{ ( 1+ |\xi|^3 |\tau^2-\gamma^2|)^{1+\epsilon}  }=& \dfrac{ |(\tau -(\gamma^2+\!\alpha^3)^{1/2})^{1/3} -\!\sqrt{\alpha}|\,|(\tau -(\gamma^2+\alpha^3)^{1/2})^{1/3} +\!\sqrt{\alpha}|}{ ( 1+ |\xi|^3 |\tau^2-\gamma^2|)^{1+\epsilon}  }\\
\lesssim & \dfrac{ \,|(\tau -(\gamma^2+\alpha^3)^{1/2})^{1/3} +\sqrt{\alpha}|}{ ( 1+ |\xi|^3 |\tau^2-\gamma^2|)^{1+\epsilon}  }\\
= &\dfrac{ \,|\tau -(\gamma^2+\alpha^3)^{1/2} +\alpha^{3/2}|}{ ( 1+ |\xi|^3 |\tau^2-\gamma^2|)^{1+\epsilon}  }\dfrac{1}{|M^2-MN+N^2|}\\
\lesssim &\dfrac{ \,|\tau -(\gamma^2+\alpha^3)^{1/2} +\alpha^{3/2}|}{ ( 1+ |\xi|^3 |\tau^2-\gamma^2|)^{1+\epsilon}  }
\\
\lesssim &\dfrac{ \,|\tau - \gamma| +| \gamma- (\gamma^2+\alpha^3)^{1/2} +\alpha^{3/2}|}{ ( 1+ |\xi|^3 |\tau^2-\gamma^2|)^{1+\epsilon}  },
\end{split}
\end{equation}
where we used the identity $M^3+N^3=(M+N)(M^2-MN+N^2)$, and the inequality \eqref{cxH-6} with $M=(\tau -(\gamma^2+\alpha^3)^{1/2})^{1/3}$ and $N=\sqrt{\alpha}$.

Thus using \eqref{x1lfc1-2}, \eqref{xH-9} and \eqref{xH-10}, we obtain
\begin{equation}\label{xH-11}
\begin{split}
\mathcal{X} \lesssim &   |\xi|\int_{l_1+\beta}^{l_2+\beta} \dfrac{ \,|\tau - \gamma|}{ ( 1+ |\xi|^3 |\tau^2-\gamma^2|)^{1+\epsilon} }d\tau+|\xi|\int_{l_1+\beta}^{l_2+\beta}\dfrac{| \gamma- (\gamma^2+\alpha^3)^{1/2} +\alpha^{3/2}|}{ ( 1+ |\xi|^3 |\tau^2-\gamma^2|)^{1+\epsilon}  } d\tau\\
\lesssim &   \dfrac{1}{|\xi|^2} +|\xi|\int_{l_1+\beta}^{l_2+\beta}\dfrac{| \gamma- (\gamma^2+\alpha^3)^{1/2} +\alpha^{3/2}|}{ ( 1+ |\xi|^3 |\tau^2-\gamma^2|)^{1+\epsilon}  } d\tau.
\end{split}
\end{equation}

In order to estimate the last integral, note that
\begin{equation}\label{xH-12}
\begin{split}
\dfrac{| \gamma- (\gamma^2+\alpha^3)^{1/2} +\alpha^{3/2}|}{ ( 1+ |\xi|^3 |\tau^2-\gamma^2|)^{1+\epsilon}  }=&\dfrac{| (\gamma+\alpha^{3/2})^2- (\gamma^2+\alpha^3) |}{ ( 1+ |\xi|^3 |\tau^2-\gamma^2|)^{1+\epsilon}  }\dfrac{1}{|\gamma+\alpha^{3/2}+(\gamma^2+\alpha^3)^{1/2}|}\\
=&\dfrac{2|\gamma|\alpha^{3/2}}{ ( 1+ |\xi|^3 |\tau^2-\gamma^2|)^{1+\epsilon}  }\dfrac{1}{|\gamma+\alpha^{3/2}+(\gamma^2+\alpha^3)^{1/2}|}.
\end{split}
\end{equation}

For $\alpha \in (0,1)$ one has $(\gamma^2+\alpha^3)^{1/2} \geq |\gamma|$.  It follows that 
\begin{equation}\label{xH-13}
\begin{split}
\gamma+\alpha^{3/2}+(\gamma^2+\alpha^3)^{1/2} \geq \gamma+\alpha^{3/2}+|\gamma| \geq \alpha^{3/2}.
\end{split}
\end{equation}
Finally using \eqref{xlfc1-2}, \eqref{xH-11}, \eqref{xH-12} and \eqref{xH-13} we deduce that
\begin{equation}\label{xH-14}
\begin{split}
\mathcal{X} 
\lesssim &   \dfrac{1}{|\xi|^2} +|\xi|\int_{l_1+\beta}^{l_2+\beta}\dfrac{| \gamma|}{ ( 1+ |\xi|^3 |\tau^2-\gamma^2|)^{1+\epsilon}  } d\tau
\lesssim   \dfrac{1}{|\xi|^2}.
\end{split}
\end{equation}
 Analogously as in the above case, we obtain $
L_1 
\lesssim  1,
$
provided $0<\epsilon <\frac{2s+1}{15}$.
\\

\noindent
{\bf Subcase d2.3) $\boxed{\Delta<0}$:}
This case is similar with the case $\Delta>0$, for the sake of completeness we make all details here. 
In this case $|\kappa_1|<2 \alpha^{3/2}$ and 
\begin{equation}\label{p00xH-5}
\begin{split}
t^2+\kappa_1 t+ \alpha^3=&(t+\dfrac{\kappa_1}2)^2+\alpha^3 -\dfrac{\kappa_1^2}{4}\\
= &(t+\beta)^2+\gamma^2,
\end{split}
\end{equation}
where $\beta= \dfrac{\kappa_1}2$ and $\gamma^2=\alpha^3-\beta^2$ consequently $\beta=\pm(\alpha^3-\gamma^2)^{1/2}$. We consider  $\beta=(\alpha^3-\gamma^2)^{1/2}$, similar argument works if $\beta=-(\alpha^3-\gamma^2)^{1/2}$. Considering a change of variable $\tau=t+\beta$ in \eqref{xH-3}, we get
\begin{equation}\label{pxH-9}
\begin{split}
\mathcal{X} \lesssim &   |\xi|\int_{l_1+\beta}^{l_2+\beta} \dfrac{ |(\tau -\beta)^{2/3} -\alpha|}{ ( 1+ |\xi|^3 |\tau^2+\gamma^2|)^{1+\epsilon}  } dt\\
\lesssim &   |\xi|\int_{l_1+\beta}^{l_2+\beta}\dfrac{ |(\tau -(\alpha^3-\gamma^2)^{1/2})^{2/3} -\alpha|}{ ( 1+ |\xi|^3 |\tau^2+\gamma^2|)^{1+\epsilon}  } dt.
\end{split}
\end{equation}

As $l_1 \leq \tau - \beta =\tau -(\alpha^3-\gamma^2)^{1/2} \leq  l_2$, analogously as above it is not difficult to see
\begin{equation}\label{pxH-10}
\begin{split}
\dfrac{ |(\tau -(\alpha^3-\gamma^2)^{1/2})^{2/3} -\alpha|}{ ( 1+ |\xi|^3 |\tau^2+\gamma^2|)^{1+\epsilon}  }=& \dfrac{ |(\tau -(\alpha^3-\gamma^2)^{1/2})^{1/3} -\sqrt{\alpha}|\,|(\tau -(\alpha^3-\gamma^2)^{1/2})^{1/3} +\sqrt{\alpha}|}{ ( 1+ |\xi|^3 |\tau^2+\gamma^2|)^{1+\epsilon}  }\\
\lesssim & \dfrac{ \,|(\tau -(\alpha^3-\gamma^2)^{1/2})^{1/3} +\sqrt{\alpha}|}{ ( 1+ |\xi|^3 |\tau^2+\gamma^2|)^{1+\epsilon}  }\\
= &\dfrac{ \,|\tau -(\alpha^3-\gamma^2)^{1/2} +\alpha^{3/2}|}{ ( 1+ |\xi|^3 |\tau^2-\gamma^2|)^{1+\epsilon}  }\dfrac{1}{|M^2-MN+N^2|}\\
\lesssim &\dfrac{ \,|\tau -(\alpha^3-\gamma^2)^{1/2} +\alpha^{3/2}|}{ ( 1+ |\xi|^3 |\tau^2+\gamma^2|)^{1+\epsilon}  }
\\
\lesssim &\dfrac{ \,|\tau + \gamma| +| -\gamma- (\alpha^3-\gamma^2)^{1/2} +\alpha^{3/2}|}{ ( 1+ |\xi|^3 |\tau^2+\gamma^2|)^{1+\epsilon}  },
\end{split}
\end{equation}
where we used the identity $M^3+N^3=(M+N)(M^2-MN+N^2)$ and the inequality \eqref{cxH-6} with $M=(\tau -(\alpha^3-\gamma^2)^{1/2})^{1/3}$ and $N=\sqrt{\alpha}$.

Thus combining \eqref{x1lfc1-2}, \eqref{pxH-9} and \eqref{pxH-10}, we obtain
\begin{equation}\label{pxH-11}
\begin{split}
\mathcal{X} \lesssim &   |\xi|\int_{l_1+\beta}^{l_2+\beta} \dfrac{ \,|\tau + \gamma|}{ ( 1+ |\xi|^3 |\tau^2+\gamma^2|)^{1+\epsilon} }dt +|\xi|\int_{l_1+\beta}^{l_2+\beta}\dfrac{| \gamma|}{ ( 1+ |\xi|^3 |\tau^2+\gamma^2|)^{1+\epsilon}  } dt\\
&+|\xi|\int_{l_1+\beta}^{l_2+\beta}\dfrac{| (\alpha^3-\gamma^2)^{1/2} -\alpha^{3/2}|}{ ( 1+ |\xi|^3 |\tau^2+\gamma^2|)^{1+\epsilon}  } dt\\
\lesssim &   \dfrac{1}{|\xi|^2} +|\xi|\int_{l_1+\beta}^{l_2+\beta}\dfrac{| (\alpha^3-\gamma^2)^{1/2} -\alpha^{3/2}|}{ ( 1+ |\xi|^3 |\tau^2+\gamma^2|)^{1+\epsilon}  } dt.
\end{split}
\end{equation}

In order to estimate the last integral, note that
\begin{equation}\label{pxH-12}
\begin{split}
\dfrac{|(\alpha^3-\gamma^2)^{1/2} -\alpha^{3/2}|}{ ( 1+ |\xi|^3 |\tau^2-\gamma^2|)^{1+\epsilon}  }=&\dfrac{| (\alpha^3-\gamma^2) -\alpha^3|}{ ( 1+ |\xi|^3 |\tau^2-\gamma^2|)^{1+\epsilon}  }\dfrac{1}{|(\alpha^3-\gamma^2)^{1/2} +\alpha^{3/2}|}\\
\leq &\dfrac{\gamma^2}{ ( 1+ |\xi|^3 |\tau^2-\gamma^2|)^{1+\epsilon}  }\dfrac{1}{\alpha^{3/2}}.
\end{split}
\end{equation}

For $\alpha \in (0,1)$, from definition of $\gamma$ one has $|\gamma| \leq \alpha^{3/2}\leq 1$. Using \eqref{xlfc1-2}, \eqref{pxH-11} and \eqref{pxH-12}, it follows that 
\begin{equation}\label{pxH-14}
\begin{split}
\mathcal{X} 
\lesssim &   \dfrac{1}{|\xi|^2} +|\xi|\int_{l_1+\beta}^{l_2+\beta}\dfrac{| \gamma|}{ ( 1+ |\xi|^3 |\tau^2-\gamma^2|)^{1+\epsilon}  } dt\\
\lesssim &   \dfrac{1}{|\xi|^2}.
\end{split}
\end{equation}
The rest follows as in {\bf Case d2.2)}.
\end{proof}

The following result shows that the bilinear estimate proved in Lemma \ref{lema2} is sharp.
\begin{proposition}\label{prop2ill}
The bilinear estimates \eqref{bil-1.1} and \eqref{bil-2.1} fail to hold for any $b \in \R$ whenever $s<-1/2$.
\end{proposition}

Now we move to prove the trilinear estimates.

\begin{proof}[Proof of Proposition \ref{prop1}] We provide detailed proof of the estimate \eqref{tlint}  only, the proof of the estimate \eqref{tlin-2t} follows analogously.

Let $b=\frac12+\epsilon$ and $b'=-\frac12+2\epsilon$. Then using duality, the estimate \eqref{tlint} is equivalent to
\begin{equation}\label{tlin-e1}
\|(uvw)_x\|_{X_{s,-\frac12+2\epsilon}}=\sup_{h\in X_{-s,\frac12-2\epsilon}}\Big|\int_{\R^2}(uvw)_xh\Big| \lesssim  \|u\|_{X_{s,\frac12+\epsilon}}\|v\|_{X^{\alpha}_{s,\frac12+\epsilon}}\|w\|_{X^{\alpha}_{s,\frac12+\epsilon}}\|h\|_{X_{-s,\frac12-2\epsilon}}.
\end{equation}

Using the relation  $\int fg =\int f\,\overline{\overline{g}} = \int \widehat{f}\;\overline{\widehat{\overline{g}}} =\int\widehat{f}\,\widecheck{g}=\int\widehat{f}(\xi)\widehat{g}(-\xi)$, one has
\begin{equation}\label{tln-e2}
\begin{split}
\Big|\int_{\R^2}(uvw)_xh\Big|&=\Big|\int_{\R^2}i\xi\int_{\xi_1+\xi_2+\xi_3=\xi\atop \tau_1+\tau_2+\tau_3=\tau} \widetilde{u}(\xi_1,\tau_1)\widetilde{v}(\xi_2,\tau_2)\widetilde{w}(\xi_3,\tau_3)\widetilde{h}(-\xi, -\tau)\Big|\\
&=\Big|\int_{\xi_1+\cdots+\xi_4=0\atop \tau_1+\cdots+\tau_4=0}\xi_4\,
\widetilde{u}(\xi_1,\tau_1)\widetilde{v}(\xi_2,\tau_2)\widetilde{w}(\xi_3,\tau_3)\widetilde{h}(\xi_4, \tau_4)\Big|.
\end{split}
\end{equation}

Using \eqref{tln-e2} in \eqref{tlin-e1}, we need to prove
\begin{equation}\label{tln-e3}
\Big|\int_{\xi_1+\cdots+\xi_4=0\atop \tau_1+\cdots+\tau_4=0}\xi_4\,
 \widetilde{u}(\xi_1,\tau_1)\widetilde{v}(\xi_2,\tau_2)\widetilde{w}(\xi_3,\tau_3)\widetilde{h}(\xi_4, \tau_4)\Big| \lesssim  \|u\|_{X_{s,\frac12+\epsilon}}\|v\|_{X^{\alpha}_{s,\frac12+\epsilon}}\|w\|_{X^{\alpha}_{s,\frac12+\epsilon}}\|h\|_{X_{-s,\frac12-2\epsilon}}.
\end{equation}

Let
\begin{equation}\label{tln-e4}
\begin{split}
&\|u\|_{X_{s,\frac12+\epsilon}}=\|\la\xi\ra^s\la\tau-\xi^3\ra^{\frac12+\epsilon}\widetilde{u}(\xi,\tau)\|_{L^2_{\xi\tau}}=:\|f_1\|_{L^2_{\xi\tau}}\\
&\|v\|_{X^{\alpha}_{s,\frac12+\epsilon}}=\|\la\xi\ra^s\la\tau-\alpha\xi^3\ra^{\frac12+\epsilon}\widetilde{v}(\xi,\tau)\|_{L^2_{\xi\tau}}=:\|f_2\|_{L^2_{\xi\tau}}\\
&\|w\|_{X^{\alpha}_{s,\frac12+\epsilon}}=\|\la\xi\ra^s\la\tau-\alpha\xi^3\ra^{\frac12+\epsilon}\widetilde{w}(\xi,\tau)\|_{L^2_{\xi\tau}}=:\|f_3\|_{L^2_{\xi\tau}}\\
&\|h\|_{X_{-s,\frac12-2\epsilon}}=\|\la\xi\ra^{-s}\la\tau-\xi^3\ra^{\frac12-2\epsilon}\widetilde{h}(\xi,\tau)\|_{L^2_{\xi\tau}}=:\|f_4\|_{L^2_{\xi\tau}}.
\end{split}
\end{equation}

From \eqref{tln-e4} and \eqref{tln-e3}, the matter reduces to proving
\begin{equation}\label{tln-e5}
\Big|\int_{\xi_1+\cdots+\xi_4=0\atop \tau_1+\cdots+\tau_4=0}
 \frac{\xi_4\,\widetilde{f_1}(\xi_1,\tau_1)\widetilde{f_2}(\xi_2,\tau_2)\widetilde{f_3}(\xi_3,\tau_3)\widetilde{f_4}(\xi_4, \tau_4)}{\la\xi_1\ra^s\la\tau_1-\xi_1^3\ra^{\frac12+\epsilon}\la\xi_2\ra^s\la\tau_2-\alpha\xi_2^3\ra^{\frac12+\epsilon}\la\xi_3\ra^s\la\tau_3-\alpha\xi_3^3\ra^{\frac12+\epsilon}\la\xi_4\ra^{-s}\la\tau_4-\xi_4^3\ra^{\frac12-2\epsilon}}\Big| \lesssim \Pi_{j=1}^4\|f_j\|_{L^2_{\xi\tau}}.
\end{equation}

So, we need to prove
\begin{equation}\label{tln-e6}
\Big|\int_{\xi_1+\cdots+\xi_4=0\atop \tau_1+\cdots+\tau_4=0}
 m(\xi_1,\tau_1,\cdots,\xi_4,\tau_4)\Pi_{j=1}^4\widetilde{f}_j(\xi_j, \tau_j)\Big| \lesssim \Pi_{j=1}^4\|f_j\|_{L^2_{\xi\tau}},
\end{equation}
where
\begin{equation}\label{tln-e7}
 m(\xi_1,\tau_1,\cdots,\xi_4,\tau_4):=\frac{\xi_4\,\la\xi_4\ra^{s}}{\la\xi_1\ra^s\la\xi_2\ra^s\la\xi_3\ra^s\la\tau_1-\xi_1^3\ra^{\frac12+\epsilon}\la\tau_2-\alpha\xi_2^3\ra^{\frac12+\epsilon}\la\tau_3-\alpha\xi_3^3\ra^{\frac12+\epsilon}\la\tau_4-\xi_4^3\ra^{\frac12-2\epsilon}}.
\end{equation}

In this way, recalling the definition of the norm $\|m\|_{[4;\R^2]}$ of the multiplier $m$,  the whole matter reduces to showing that
\begin{equation}\label{tln-e8}
\|m\|_{[4;\R^2]}\lesssim 1.
\end{equation}

Note that $-\xi_4 = \xi_1+\xi_2+\xi_3\implies \la\xi_4\ra\leq\la\xi_1\ra+\la\xi_2\ra+\la\xi_3\ra$. Therefore, considering $s+1=s_0+s_1$ with $s_1\geq 0$, one can obtain
\begin{equation}\label{xi-4}
 \xi_4\la\xi_4\ra^s\leq\la\xi_4\ra^{s+1}\lesssim\la\xi_4\ra^{s_0}\big(\la\xi_1\ra^{s_1}+\la\xi_2\ra^{s_1}+\la\xi_3\ra^{s_1}\big).
\end{equation}

From \eqref{xi-4} and \eqref{tln-e7}, we get
\begin{equation}\label{mult-1}
\begin{split}
 m\leq &\frac{\la\xi_4\ra^{s_0}\la\xi_1\ra^{s_1}}{\la\xi_1\ra^s\la\xi_2\ra^s\la\xi_3\ra^s\la\tau_1-\xi_1^3\ra^{\frac12+\epsilon}\la\tau_2-\alpha\xi_2^3\ra^{\frac12+\epsilon}\la\tau_3-\alpha\xi_3^3\ra^{\frac12+\epsilon}\la\tau_4-\xi_4^3\ra^{\frac12-2\epsilon}}\\
 &+ \frac{\la\xi_4\ra^{s_0}\la\xi_2\ra^{s_1}}{\la\xi_1\ra^s\la\xi_2\ra^s\la\xi_3\ra^s\la\tau_1-\xi_1^3\ra^{\frac12+\epsilon}\la\tau_2-\alpha\xi_2^3\ra^{\frac12+\epsilon}\la\tau_3-\alpha\xi_3^3\ra^{\frac12+\epsilon}\la\tau_4-\xi_4^3\ra^{\frac12-2\epsilon}}\\
 &+ \frac{\la\xi_4\ra^{s_0}\la\xi_3\ra^{s_1}}{\la\xi_1\ra^s\la\xi_2\ra^s\la\xi_3\ra^s\la\tau_1-\xi_1^3\ra^{\frac12+\epsilon}\la\tau_2-\alpha\xi_2^3\ra^{\frac12+\epsilon}\la\tau_3-\alpha\xi_3^3\ra^{\frac12+\epsilon}\la\tau_4-\xi_4^3\ra^{\frac12-2\epsilon}}\\
 =:&J_1+J_2+J_3.
 \end{split}
\end{equation}

If we take $s_0=\frac12$, then $s_1=s+\frac12$, so $J_i$, $i=1,2,3$ can be written as
\begin{equation}\label{J1}
\begin{split}
J_1&=\frac{\la\xi_4\ra^{\frac12}\la\xi_1\ra^{\frac12}}{\la\xi_2\ra^s\la\xi_3\ra^s\la\tau_1-\xi_1^3\ra^{\frac12+\epsilon}\la\tau_2-\alpha\xi_2^3\ra^{\frac12+\epsilon}\la\tau_3-\alpha\xi_3^3\ra^{\frac12+\epsilon}\la\tau_4-\xi_4^3\ra^{\frac12-2\epsilon}}\\
& =  \frac{\la\xi_1\ra^{\frac12}}{\la\tau_1-\xi_1^3 \ra^{\frac12+\epsilon} \la\xi_2\ra^s \la\tau_2-\alpha\xi_2^3\ra^{\frac12+\epsilon}} \frac{\la\xi_4\ra^{\frac12}}{\la\tau_4-\xi_4^3\ra^{\frac12-2\epsilon}\la\xi_3\ra^s\la\tau_3-\alpha\xi_3^3\ra^{\frac12+\epsilon}}\\
&\leq m_1(\xi_1, \tau_1, \xi_2,\tau_2)\,m_1(\xi_4,\tau_4,\xi_3,\tau_3),
\end{split}
\end{equation}
\begin{equation}\label{J2}
\begin{split}
J_2&=\frac{\la\xi_4\ra^{\frac12}\la\xi_2\ra^{\frac12}}{\la\xi_1\ra^s\la\xi_3\ra^s\la\tau_1-\xi_1^3\ra^{\frac12+\epsilon}\la\tau_2-\alpha\xi_2^3\ra^{\frac12+\epsilon}\la\tau_3-\alpha\xi_3^3\ra^{\frac12+\epsilon}\la\tau_4-\xi_4^3\ra^{\frac12-2\epsilon}}\\
&=\frac{\la\xi_4\ra^{\frac12}}{\la\tau_4-\xi_4^3\ra^{\frac12-2\epsilon}\la\xi_3\ra^s\la\tau_3-\alpha\xi_3^3\ra^{\frac12+\epsilon}} \frac{\la\xi_2\ra^{\frac12}}{\la\tau_2-\alpha\xi_2^3\ra^{\frac12+\epsilon}\la\xi_1\ra^s\la\tau_1-\xi_1^3\ra^{\frac12+\epsilon}
}\\
&\leq m_1(\xi_4,\tau_4,\xi_3,\tau_3)\,m_2(\xi_2,\tau_2,\xi_1,\tau_1),
\end{split}
\end{equation}
\begin{equation}\label{J3}
\begin{split}
J_3&=\frac{\la\xi_4\ra^{\frac12}\la\xi_3\ra^{\frac12}}{\la\xi_1\ra^s\la\xi_2\ra^s\la\tau_1-\xi_1^3\ra^{\frac12+\epsilon}\la\tau_2-\alpha\xi_2^3\ra^{\frac12+\epsilon}\la\tau_3-\alpha\xi_3^3\ra^{\frac12+\epsilon}\la\tau_4-\xi_4^3\ra^{\frac12-2\epsilon}}\\
&=\frac{\la\xi_4\ra^{\frac12}}{\la\tau_4-\xi_4^3\ra^{\frac12-2\epsilon}\la\xi_2\ra^s\la\tau_2-\alpha\xi_2^3\ra^{\frac12+\epsilon}}\frac{\la\xi_3\ra^{\frac12}}{\la\tau_3-\alpha\xi_3^3\ra^{\frac12+\epsilon}\la\xi_1\ra^s\la\tau_1-\xi_1^3\ra^{\frac12+\epsilon}}\\
&\leq m_1(\xi_4,\tau_4,\xi_2,\tau_2)\,m_2(\xi_3,\tau_3,\xi_1,\tau_1).
\end{split}
\end{equation}

With all what we did above, using comparison principle, permutation and composition properties (see respectively Lemmas 3.1, 3.3 and 3.7 in \cite{Tao}), it is enough to bound $\|m_j\|_{[3;\R^2]}$, $j=1,2$, where
\begin{equation}\label{m1}
m_1(\xi_1, \tau_1, \xi_2,\tau_2)=\frac{\la\xi_1\ra^{\frac12}}{\la\tau_1-\xi_1^3 \ra^{\frac12-2\epsilon} \la\xi_2\ra^s \la\tau_2-\alpha\xi_2^3\ra^{\frac12+\epsilon}},
\end{equation}
\begin{equation}\label{m12}
m_2(\xi_1, \tau_1, \xi_2,\tau_2)=\frac{\la\xi_1\ra^{\frac12}}{\la\tau_1-\alpha\xi_1^3 \ra^{\frac12-2\epsilon} \la\xi_2\ra^s \la\tau_2-\xi_2^3\ra^{\frac12+\epsilon}}.
\end{equation}

 With the argument presented in \cite{Tao}, proving 
$\|m_j\|_{[3;\R^2]}\lesssim 1$, $j=1,2$ is equivalent respectively to showing the following bilinear estimates
\begin{equation}\label{bil-1}
\|uv\|_{L^2(\R^2)}\lesssim \|u\|_{X_{-\frac12, \frac12-2\epsilon}}\|v\|_{X^{\alpha}_{s, \frac12+\epsilon}},
\end{equation}
and
\begin{equation}\label{bil-2}
\|uv\|_{L^2(\R^2)}\lesssim \|u\|_{X^{\alpha}_{-\frac12, \frac12-2\epsilon}}\|v\|_{X_{s, \frac12+\epsilon}}.
\end{equation}
This equivalence can be proved using duality with similar calculations as the ones used to get \eqref{tln-e6}.

Recall that the estimates \eqref{bil-1} and \eqref{bil-2} for $s>- \frac12$ are already proved in Lemma \ref{lema2}, and this completes the proof of \eqref{tlint}.
\end{proof}

\begin{remark}
For $\alpha=1$, one has  $m_1=m_2$  and in this case the bilinear estimate \eqref{bil-2} and thus  the norm of $m_2$   is proven to be bounded in \cite{Tao} for all $s\geq \frac14$. Our interest here is to consider $0<\alpha<1$.
\end{remark}

\subsection{Proof of the local well-posedness result}


In this subsection  we use the linear estimates recorded in Section \ref{Prel-e} and the trilinear estimates from Proposition \ref{prop1} to  sketch  proof of Theorem \ref{loc-sys}. 

\begin{proof}[Proof of Theorem \ref{loc-sys}]
 We define spaces $Z_{s,b}:=X_{s, b}\times X^{\alpha}_{s, b}$ and $Y^s:= H^{s}\times H^{s}$ with norms  $\|(v, w)\|_{X_{s, b}\times X^{\alpha}_{s, b}} := \max \{\|v\|_{X_{s,b}}, \|w\|_{X^{\alpha}_{s,b}}\}$ and similar for $Y^s$.  Let $a>0$ and consider a ball in $Z_{s, b}$ given by
\begin{equation}\label{sy-ball}
 \mathcal{X}_s^{a} = \bigl\{ (v, w) \in Z_{s, b};\:\|(v, w)\|_{Z_{s, b}} < a\bigl\}.
 \end{equation}

As we are interested in finding local solution to the IVP \eqref{ivp-sy}, we define the following application with the use of cut-off functions
\begin{equation}\label{contrac-s2}
\begin{cases}
\Phi_{\phi}[v, w](t):= \psi_1(t)U(t)\phi - \psi_{\delta}(t)\displaystyle\int_{0}^{t}U(t-t')\p_x(vw^2)(t')\,dt',\\
\Psi_{\psi}[v, w](t):= \psi_1(t)U^{\alpha}(t)\psi - \psi_{\delta}(t)\displaystyle\int_{0}^{t}U^{\alpha}(t-t')\p_x(v^2w)(t')\,dt'.
\end{cases}
\end{equation}

 We will show that, there exist $a>0$ and $\delta >0$ such that the application $\Phi\times \Psi$ maps
 $\mathcal{X}_s^{a}$ into $\mathcal{X}_s^{a}$ and is a contraction.

We start  estimating the first component $\Phi$. Using estimates \eqref{lin.1}, \eqref{nlin.1} from Lemma \ref{lemma1},  we have for $s>- \frac12$ and $\theta := 1-b+b' >0$,
\begin{equation}\label{sy-1}
\|\Phi\|_{X_{s,b}}\leq C\|\phi\|_{H^{s}} +C\delta^{\theta}\|\p_x(vw^2)\|_{X_{s,b'}}.
\end{equation}

Now, using the   trilinear estimates from Proposition \ref{prop1}, one obtains
\begin{equation}\label{til-m22}
\|\p_x(vw^2\|_{X_{s,b'}}\leq C\|v\|_{X_{s,b}}\|w\|_{X_{s,b}^{\alpha}}^2.
\end{equation}

Inserting  \eqref{til-m22} in \eqref{sy-1}, one gets
\begin{equation}\label{phi-11}
\begin{split}
 \|\Phi\|_{X_{s,b}}
 &\leq C\|\phi\|_{H^s} +C\delta^{\theta}\|v\|_{X_{s,b}}\|w\|_{X_{s,b}^{\alpha}}^2\\
& \leq C\|(\phi, \psi)\|_{Y^s} +C\delta^{\theta}\|(v, w)\|_{Z_{s,b}}^3.
\end{split}
\end{equation}

In an analogous manner, it is not difficult to obtain
\begin{equation}\label{sy-2}
\|\Psi\|_{X_{s,b}^{\alpha}}\leq C\|(\phi, \psi)\|_{Y^s} +C\delta^{\theta}\|(v, w)\|_{Z_{s,b}}^3.
\end{equation}

Therefore, from \eqref{phi-11} and \eqref{sy-2}, we obtain
\begin{equation}\label{sy-3}
\|(\Phi, \Psi)\|_{Z_{s,b}}\leq C\|(\phi, \psi)\|_{Y^s} +C\delta^{\theta}\|(v, w)\|_{Z_{s,b}}^3.
\end{equation}

Let us choose $a = 2C\|(\phi, \psi)\|_{Y^s}$,  then from \eqref{sy-3}, we get
\begin{equation}\label{sy-4}
\|(\Phi, \Psi)\|_{Z_{s,b}}\leq \frac a2 + C\delta^{\theta} a^3.
\end{equation}

Now, if we take $\delta>0$ such that $C\delta^{\theta} a^2<\frac12$, then \eqref{sy-4} yields
\begin{equation}\label{sy-5}
\|(\Phi, \Psi)\|_{Z_{s,b}}\leq \frac a2 + \frac a2 = a.
\end{equation}

Therefore, the application $\Phi\times\Psi$ maps $\mathcal{X}_s^{a}$ into $\mathcal{X}_s^{a}$. With the similar technique, one can easily show that $\Phi\times\Psi$ is a contraction. Hence by a standard argument one can prove that the IVP \eqref{ivp-sy} is locally well-posed for initial data $(\phi, \psi)\in Y^s$ for any $s>- \frac12$. Moreover, from \eqref{sy-5} and the choice of $a$, one has
\begin{equation}\label{sy-6}
\|(\Phi, \Psi)\|_{Z_{s,b}}\leq C\|(\phi, \psi)\|_{Y^s}.
\end{equation}
 The rest of the proof follows standard argument, so we omit the details.
 \end{proof}

\section{Failure of bilinear and trilinear estimates, and Ill-posedness}\label{ill-posd-T}
In this section we prove the failure of the  trilinear estimates that are crucial in the argument we used to obtain local well-posedness result using Bourgain's spaces.  Also we will prove the failure of the bilinear estimate without derivative that plays a fundamental role in the proof of the trilinear estimates.  Finally, we prove ill-posedness result by showing that the application data-solution is not $C^3$ at the origin for initial data with Sobolev regularity below $-\frac12$. To obtain these negative results we construct  counter examples exploiting the lack of cancellation in the resonance relation when $0<\alpha<1$.

\subsection{Failure of trilinear estimates}
We start by recording two elementary results about convolution that we use to prove the failure of trilinear estimate.

\begin{lemma}\label{conv-lema1}
Let $R_j:=[a_j, b_j]$, $j=1,\cdots,n$ be intervals of size $|R_j| = b_j-a_j=N$. Then
\begin{equation}\label{conv-e1}
\|\chi_{R_1}\convolution\chi_{R_2}\convolution\cdots\convolution\chi_{R_n}\|_{L^2}\sim N^{n-\frac12}=|R_j|^{n-\frac12}.
\end{equation}
\begin{proof} First note that
\begin{equation}
\widehat{\chi_{R_j}}(\xi) \sim \int_{a_j}^{b_j}e^{-i\xi x}dx = \frac2\xi e^{-i\xi\frac{a_j+b_j}{2}}\sin \big(\frac{b_j-a_j}2\xi\big).
\end{equation}
Now, using Plancherel's identity
\begin{equation}
\|\chi_{R_1}\convolution\chi_{R_2}\convolution\cdots\convolution\chi_{R_n}\|_{L^2}\sim \|\dfrac{\sin^n(\frac N2\xi)}{\xi^n}\|_{L^2} \sim N^{n-\frac12}.
\end{equation}
\end{proof}
\end{lemma}

\begin{lemma}\label{conv-lema2}
Let $R_j:=[a_j, b_j]\times[c_j, d_j]\subset\R^2$, $j=1,\cdots,n$ be rectangles such that $ b_j-a_j=N$ and $d_j-c_j=M$. Then
\begin{equation}\label{conv-e2}
\|\chi_{R_1}\convolution\chi_{R_2}\convolution\cdots\convolution\chi_{R_n}\|_{L^2(\R^2)}\sim (NM)^{n-\frac12}=|R_j|^{n-\frac12}.
\end{equation}
\end{lemma}
\begin{proof} The proof follows by using Fubini's Theorem 
$$
\widetilde{\chi_{R_j}}(\xi, \tau)=\widehat{\chi_{[a_j, b_j]}}(\xi)\widehat{\chi_{[c_j, d_j]}}(\tau),
$$
followed by Lemma \ref{conv-lema1}. 
\end{proof}

Now  prove the result on  the failure of the trilinear estimates stated in Proposition \ref{prop1ill}.

\begin{proof}[Proof of Proposition \ref{prop1ill}] 
We give details of the proof for the failure of \eqref{tlint-m1} whenever $s<-\frac12$ and $0<\alpha<1$. The proof of \eqref{tlint-m2} follows analogously. 

In fact, we will prove a general result showing that
\begin{equation}
\label{tril-fail2}
\|\p_x(u_1u_2u_3)\|_{X_{s, b'}}\lesssim\|u_1\|_{X_{s, b}}\|u_2\|_{X_{s, b}^{\alpha}}\|u_3\|_{X_{s, b}^{\alpha}}
\end{equation}
fails to hold whenever $s<-\frac12$ by constructing a counter example.

Using definition of the $X_{s,b}$-norm  and Plancherel's identity,  the estimate \eqref{tril-fail2} is equivalent to proving
\begin{equation}\label{tril-fail3}
\| \mathcal{T}_s(f,g,h) \|_{L^2_{\xi} L^2_{\tau}} \lesssim \|f\|_{L^2(\R^2)} \|g\|_{L^2(\R^2)}\|h\|_{L^2(\R^2)},
\end{equation}
where 
\begin{equation}\label{trill-f4}
\mathcal{T}_s(f,g,h):= \Big\|\la\xi\ra^s\la\tau-\xi^3\ra^{b'}\xi\int_{\R^4} \dfrac{\widetilde{f}(\xi_1, \tau_1)\widetilde{g}(\xi_2, \tau_2) \widetilde{h}(\xi_3, \tau_3)}{\langle \xi_1 \rangle^s\la \tau_1-\xi_1^3 \ra^{b} \langle \xi_2 \rangle^s\la \tau_2-{\alpha}\xi_2^3 \ra^{b}\langle \xi_3 \rangle^s\la \tau_3-\alpha\xi_3^3 \ra^{b}}\Big\|_{L^2_{\xi\tau}},
\end{equation}
with $\xi_3=\xi-\xi_1-\xi_2$, $\tau_3=\tau-\tau_1-\tau_2$, and 
$$\widetilde{f}(\xi, \tau):=\la\xi\ra^s\la\tau-\xi^3\ra^b\widetilde{u_1}(\xi, \tau)$$
$$\widetilde{g}(\xi, \tau):=\la\xi\ra^s\la\tau-{\alpha}\xi^3\ra^b\widetilde{u_2}(\xi, \tau)$$
$$\widetilde{h}(\xi, \tau):=\la\xi\ra^s\la\tau-\alpha\xi^3\ra^b\widetilde{u_3}(\xi, \tau).$$

Let $c_1$, $c_2$ and $c_3$ be three constants satisfying
\begin{equation}\label{cond-const}
\begin{split}
&c_1+c_2+c_3=1\\
&c_1^3+\alpha c_2^3+{\alpha}c_3^3 =1.
\end{split}
\end{equation}
Consider three rectangles $R_1$, $R_2$ and $R_3$ with centres at $(c_1N, (c_1N)^3)$, $(c_2N, \alpha(c_2N)^3)$ and $(c_3N, \alpha(c_3N)^3)$ respectively, and each with dimension $N^{-2}\times 1$. Now, consider
$$\widetilde{f} = \chi_{R_1}, \qquad \widetilde{g} = \chi_{R_2}, \qquad \widetilde{h} = \chi_{R_3}.$$

It is easy to show that 
\begin{equation}\label{norms-fgh}
\|f\|_{L^2(\R^2)} =\|g\|_{L^2(\R^2)}
=\|h\|_{L^2(\R^2)}=N^{-1}.
\end{equation}
Also, 
\begin{equation}\label{tri-fail5}
| \xi_1-c_1N|\leq \frac{N^{-2}}2, \qquad |\tau_1-(c_1N)^3|\leq \frac12
\end{equation}
and
\begin{equation}\label{tri-fail6}
|\xi_j-c_jN|\leq \frac{N^{-2}}2, \qquad |\tau_j-\alpha(c_jN)^3|\leq \frac12, \qquad j=2,3.
\end{equation}

For $(\xi, \tau)\in R_1+R_2+R_3$, using \eqref{cond-const}, \eqref{tri-fail5} and \eqref{tri-fail6}, it is easy to prove that
\begin{equation}\label{tri-fail71}
\begin{split}
 |\xi-N|\lesssim N^{-2}, \qquad {\textrm{and}}\qquad 
 |\tau-\xi^3|\lesssim 1.
 \end{split}
\end{equation}
In fact, the first estimate in \eqref{tri-fail71} is a consequence of
$$
|\xi-N|= |\xi_1- Nc_1+\xi_2-N c_1+ \xi_3-N c_3|
$$
and the second estimate  is a consequence of
\begin{equation*}
\begin{split}
|\tau-\xi^3|=&|\tau-N^3+N^3-\xi^3|\\
=& |\tau_1- (Nc_1)^3+\tau_2- \alpha(Nc_2)^3+\tau_3-\alpha (Nc_3)^3+ (N-\xi)(N^2+N \xi+ \xi^2)|.
\end{split}
\end{equation*}

As a result, for $(\xi, \tau)\in R_1+R_2+R_3$ we have $|\xi|\sim N$ and $\la\tau-\xi^3\ra\sim 1$. Also, it is not difficult to show  that $|\xi_j|\sim N$, $j=1,2,3$, and $\la\tau_1-\xi_1^3\ra\sim 1$ and $\la\tau_j-\alpha\xi_j^3\ra\sim 1$, $j=2,3$.

Finally, using  above considerations in \eqref{tril-fail3}, we obtain
\begin{equation*}\label{tri-fail8}
\Big\|N^{s+1}\int_{\R^4} \dfrac{\chi_{R_1}(\xi_1,\tau_1)\chi_{R_2}(\xi_2,\tau_2)\chi_{R_3}(\xi_3,\tau_3)}{N^{3s}}\Big\|_{L^2_{\xi\tau}} \lesssim N^{-3}
\end{equation*}
which is equivalent to 
\begin{equation}\label{tri-fail9}
N^{-2s+1}\big\| \chi_{R_1}\convolution\chi_{R_2}\convolution\chi_{R_3}\big\|_{L^2_{\xi\tau}} \lesssim N^{-3}.
\end{equation}

Now, using estimate \eqref{conv-e2} from Lemma \ref{conv-lema2} with $n=3$ and $|R_j|=N^{-2}$, the estimate \eqref{tri-fail9} yields
\begin{equation}\label{tri-fail10}
N^{-2s+1}N^{-5} \lesssim N^{-3} \Longleftrightarrow N^{-2s-1}\lesssim 1.
\end{equation}

If we choose $N$ large, the estimate \eqref{tri-fail10} fails to hold whenever $s<-\frac12$, and this completes the proof of the theorem.
\end{proof}

\subsection{Failure of bilinear estimates}

Now we  prove the failure of bilinear estimates \eqref{bil-1.1} and \eqref{bil-2.1} stated in Lemma \ref{lema2} whenever $s<-\frac12$. The following proposition shows that the bilinear estimate \eqref{bil-1.1} fails for  $s<-\frac12$. The proof of the failure of  \eqref{bil-2.1} is analogous.
\begin{proposition}\label{gail-b1}
Let $0<\alpha<1$ and $s<-\frac12$, then  the following bilinear estimate
\begin{equation}\label{fail-be1}
\|uv\|_{L^2(\R^2)} \lesssim \|u\|_{X_{-\frac12, \frac12-2\epsilon}^{\alpha}}\|v\|_{X_{s, \frac12+\epsilon}}
\end{equation}
fails to hold.
\end{proposition}
\begin{proof}
  Using Plancherel's identity, 
the estimate \eqref{fail-be1} is equivalent to showing that
\begin{equation}\label{fail-be2}
\mathcal{B}_s(f,g):=\Big\| \int_{\R^2} \dfrac{\la \xi_2 \ra^{1/2}\widetilde{f}(\xi_2, \tau_2)\widetilde{g}(\xi_1, \tau_1) }{\langle \xi_1 \rangle^s\la \tau_1-\alpha\xi_1^3 \ra^{\frac12+\epsilon} \la \tau_2-\xi_2^3 \ra^{\frac12-2\epsilon} } d\xi_1 d\tau_1 \Big\|_{L^2_{\xi\tau}(\R^2)} \lesssim \|f\|_{L^2(\R^2)} \|g\|_{L^2(\R^2)},
\end{equation}
where 
  $\widetilde{f}(\xi,\tau)=\la\xi\ra^{-\frac12}\la\tau-\alpha\xi^3\ra^{\frac12-2\epsilon}\widetilde{u}(\xi,\tau)$,  $\widetilde{g}(\xi,\tau)=\la\xi\ra^s\la\tau-\xi^3\ra^{\frac12+\epsilon} \widetilde{v}(\xi,\tau)$, $\xi_2=\xi-\xi_1$ and $\tau_2=\tau-\tau_1$.

We will construct functions $f$ and $g$ for which the estimate \eqref{fail-be2} fails to hold when $s<-\frac12$. 

Consider two rectangles $R_1$ and  $R_2$  with centres respectively at $(N, \alpha N^3)$, and  $(N, N^3)$, and each with dimension $N^{-2}\times 1$. Now, consider $f$ and $g$ defined via their Fourier transform
$$\widetilde{g} = \chi_{R_1},  \qquad \widetilde{f}= \chi_{R_2}.$$

It is easy to see that 
\begin{equation}\label{fg-norms}
\|f\|_{L^2(\R^2)} =\|g\|_{L^2(\R^2)}=N^{-1}.
\end{equation}
Also, 
\begin{equation}\label{fail-be5}
| \xi_1-N|\leq \frac{N^{-2}}2, \qquad |\tau_1-\alpha N^3|\leq \frac12
\end{equation}
and
\begin{equation}\label{fail-be6}
|\xi_2-N|\leq \frac{N^{-2}}2, \qquad |\tau_2-N^3|\leq \frac12.
\end{equation}

Using \eqref{fail-be5}, it is easy to prove that,
\begin{equation*}\label{tri-fail7}
 | \tau_1-\alpha\xi_1^3| \leq | \tau_1-\alpha N^3| +|\alpha N^3-\alpha\xi_1^3|\lesssim 1,
\end{equation*}
and similarly using \eqref{fail-be6}
\begin{equation*}\label{tri-fail7}
 | \tau_2-\xi_2^3| \leq | \tau_2- N^3| +|N^3-\xi_2^3|\lesssim 1.
\end{equation*}
Consequently, we have  $|\xi_j|\sim N$ and $\la\tau_1-\alpha\xi_1^3\ra\sim 1$ and $\la\tau_2-\xi_2^3\ra\sim 1$.

With these considerations, we get from \eqref{fail-be2}
\begin{equation}\label{fail-be8}
\begin{split}
\mathcal{B}_s(f,g) 
&\sim \Big\| N^{\frac12-s}\int_{\R^2} \chi_{R_1}(\xi_1, \tau_1)\chi_{R_2}(\xi_2, \tau_2) d\xi_1 d\tau_1 \Big\|_{L^2_{\xi\tau}(\R^2)}\\
&\sim N^{\frac12-s} \|\chi_{R_1}\convolution\chi_{R_2}\|_{L^2(\R^2)}\\
&= N^{\frac12-s}|R_j|^{2-\frac12} = N^{-\frac52-s}.
\end{split}
\end{equation}

Now, using  \eqref{fg-norms} and  \eqref{fail-be8} in \eqref{fail-be2}, 
\begin{equation}\label{fail-be9}
 N^{-\frac52-s} \lesssim N^{-2} \Longleftrightarrow N^{-\frac12-s}\lesssim 1.
\end{equation}

If we choose $N$ large, the estimate \eqref{fail-be9} fails to hold whenever $s<-\frac12$, and this completes the proof of the theorem.
\end{proof}

\subsection{Ill-posedness result}
In this subsection
, we  consider the ill-posedness issue for the IVP \eqref{ivp-sy}.
We start with the following result which is the main ingredient to prove the ill-posedness result stated in Theorem \ref{mainTh-ill1.1}.

\begin{proposition}\label{prop4.1}
Let  $s<-\frac12$, $0<\alpha<1$ and $T>0$. Then there does not exist  a space $X_T^s\times X_T^s$ continuously embedded in $C([0, T]; H^s(\R))\times C([0, T]; H^s(\R))$ such that the estimates
\begin{equation}\label{eq4.01}
\|(U(t)\phi, U^{\alpha}(t)\psi)\|_{X_T^s\times X_T^s}\lesssim \|(\phi, \psi)\|_{H^s\times H^s},
\end{equation}
and
\begin{equation}\label{eq4.02}
\Big\|(\Phi_3, \Psi_3)\Big\|_{X_T^s\times X_T^s}\lesssim \|(v, w)\|_{X_T^s\times X^s_T}^3,
\end{equation}
hold true, where 
\begin{equation*}\label{def-6.1}
\begin{split}
&\Phi_3:= -6\int_0^tU(t-t')\p_x[vw^2](t')dt'\\
&\Psi_3:= -6\int_0^tU^{\alpha}(t-t')\p_x[v^2w](t')dt'.
\end{split}
\end{equation*}
\end{proposition}

\begin{proof}
The proof follows a contradiction argument. If possible, suppose that there exists a space $X_T^s\times X_T^s$ that is continuously embedded in $C([0, T]; H^s(\R))\times C([0, T]; H^s(\R))$ such that the estimates \eqref{eq4.01} and \eqref{eq4.02} hold true.  If we consider $v= U(t)\phi$ and $w= U^{\alpha}(t)\psi$, then from \eqref{eq4.01} and \eqref{eq4.02}, we get
\begin{equation}\label{eq4.03}
\|\Phi_3\|_{H^s}\leq \Big\|(\Phi_3, \Psi_3)\Big\|_{H^s\times H^s}\lesssim \|(\phi, \psi)\|_{H^s\times H^s}^3,
\end{equation}
 where
 \begin{equation}\label{def-6.2}
\begin{split}
&\Phi_3:= -6\int_0^tU(t-t')\p_x\big[U(t')\phi(U^{\alpha}(t')\psi)^2\big]dt'\\
&\Psi_3:= -6\int_0^tU^{\alpha}(t-t')\p_x\big[(U(t')\phi)^2U^{\alpha}(t')\psi\big]dt'.
\end{split}
\end{equation}

The main idea to complete the proof  is to find an appropriate initial data $(\phi, \psi)$ for which the estimate \eqref{eq4.03} fails to hold whenever $s<-\frac12$.

Let $N\gg 1$, $\gamma=\gamma(N)$ to be chosen later such that $0< \gamma \ll 1$. Let us consider
$\phi$, $\psi$ defined via Fourier transform
\begin{equation}\label{phi-hat1}
\widehat{\phi}(\xi)=\gamma^{-\frac12}N^{-s} \chi_{R_1}(\xi)
\end{equation}
and
\begin{equation}\label{psi-hat1}
\widehat{\psi}(\xi)=\gamma^{-\frac12}N^{-s}\big[\chi_{R_2}(\xi)+\chi_{R_3}(\xi)\big]=\gamma^{-\frac12}N^{-s}P(\xi),
\end{equation}
where \begin{equation}\label{Rjs}
R_j(\xi)=\{\xi: |\xi-  c_jN|\leq \gamma\}, \qquad j=1,2,3
\end{equation}
with constants $c_j$ satisfying
\begin{equation}\label{cond-cs}
c_1+ c_2+  c_3=1, \qquad \quad c_1^3+\alpha c_2^3+\alpha c_3^3=1.
\end{equation}

In view of the definition of unitary groups $U$ and  $U^{\alpha}$, and choices  of  $\phi$ and $\psi$ respectively in \eqref{phi-hat1} and \eqref{psi-hat1}, we have
\begin{equation}\label{psi1-hat}
\widehat{U^{\alpha}(t)\psi}(\xi)=e^{i\alpha t\xi^3}\widehat{\psi}(\xi) = \gamma^{-\frac12}N^{-s}e^{i\alpha t\xi^3}P(\xi),
\end{equation}
and
\begin{equation}\label{psi1-hatU}
\widehat{U(t)\phi}(\xi)=e^{i t\xi^3}\widehat{\phi}(\xi) = \gamma^{-\frac12}N^{-s}e^{i t\xi^3}\chi_{R_1}(\xi).
\end{equation}
Let  $\Phi_1:=U(t)\phi$ and $\Psi_1:=U^{\alpha}(t)\psi$. Taking Fourier transform in \eqref{def-6.2}, we obtain
\begin{equation}\label{FT-phi3}
\begin{split}
\widehat{\Phi_3}(\xi, t) = -6\int_0^te^{i(t-t')\xi^3}i\xi\int_{\R^2}\widehat{\Psi_1}(\xi-\xi_1-\xi_2)\widehat{\Psi_1}(\xi_2)\widehat{\Phi_1}(\xi_1)d\xi_1 d\xi_2dt'\\
\widehat{\Psi_3}(\xi, t) = -6\int_0^te^{i(t-t')\alpha\xi^3}i\xi\int_{\R^2}\widehat{\Phi_1}(\xi-\xi_1-\xi_2)\widehat{\Phi_1}(\xi_2)\widehat{\Psi_1}(\xi_1)d\xi_1 d\xi_2dt'.
\end{split}
\end{equation}

Let $\xi_3=\xi-\xi_1-\xi_2$.  For the choices of $\phi$ and $\psi$, using the relations \eqref{psi1-hat} and \eqref{psi1-hatU} in \eqref{FT-phi3}, we get
\begin{equation}\label{FT-phi3.2}
\begin{split}
\widehat{\Phi_3}(\xi, t)& =  -6e^{it\xi^3}i\xi\int_0^te^{-it'\xi^3}\int_{\R^2} \gamma^{-\frac32}N^{-3s}e^{i t'(\xi_1^3+\alpha\xi_2^3+\alpha\xi_3^3)}\chi_{R_1}(\xi_1) P(\xi_2)P(\xi_3)d\xi_1 d\xi_2dt'\\
&=-6e^{it\xi^3}i\xi\gamma^{-\frac32}N^{-3s}\int_{\R^2}\chi_{R_1}(\xi_1) P(\xi_2)P(\xi_3)\int_0^te^{i t'(\xi_1^3+\alpha\xi_2^3+\alpha\xi_3^3-\xi^3)}dt'd\xi_1 d\xi_2\\
&=-6e^{it\xi^3}i\xi\gamma^{-\frac32}N^{-3s}\int_{\R^2}\chi_{R_1}(\xi_1) P(\xi_2)P(\xi_3)\frac{e^{itA}-1}{iA}d\xi_1 d\xi_2,
\end{split}
\end{equation}
where $A:= \xi_1^3+\alpha\xi_2^3+\alpha\xi_3^3-\xi^3$. 

To continue with the proof we need the following  lemmas.

\begin{lemma}\label{Lema-c3.1}
Let $N\gg 1$, $0<\gamma\ll 1$, $|\xi-N|\leq \gamma$, $\xi_1 \in R_1$ and $R_j$, $j=1,2,3$ as defined in \eqref{Rjs}, then 
$$(\xi_1, \xi_2, \xi_3)\in R_{1}\times R_{2}\times R_{3},$$ or
$$(\xi_1, \xi_2, \xi_3)\in R_{1}\times R_{3}\times R_{2}.$$
\end{lemma}
\begin{proof}
We prove it using contradiction argument. If possible, suppose that $(\xi_1, \xi_2, \xi_3)\notin R_{1}\times R_{2}\times R_{3}$ and $(\xi_1, \xi_2, \xi_3)\notin R_{1}\times R_{3}\times R_{2}$. It means, one can get two members of $\{\xi_1, \xi_2, \xi_3\}$ belonging to the same rectangle. For simplicity of exposition, let us consider the case when $\xi_1\in R_1$ and $\xi_2, \xi_3\in R_2$ (the other case can be handled analogously.) Therefore, by definition
\begin{equation*}\label{lm-c4.1}
|\xi_1-c_1N|\leq \gamma, \qquad |\xi_2-c_2N|\leq \gamma \qquad {\textrm{and}}\qquad |\xi_3-c_2N|\leq \gamma.
\end{equation*}
Noting that $\xi=\xi_1+\xi_2+\xi_3$ and using the first relation in \eqref{cond-cs}, we get
\begin{equation*}\label{lem-c4.2}
\begin{split}
|\xi-N|&=|\xi_1+\xi_2+\xi_3-(c_1N+c_2N+c_3N)|\\
&= |\xi_1-c_1N +\xi_2-c_2N +\xi_3-c_2N +(c_2-c_3)N|\\
&\geq |c_2-c_3|N-3\gamma\\
&\geq \frac12|c_2-c_3|N,
\end{split}
\end{equation*}
if we choose $\gamma \leq \frac16|c_2-c_3|N$. From last expression we obtain that $|\xi-N|\gtrsim N$ which contradicts the fact that $|\xi-N|\leq \gamma$.
\end{proof}

\begin{lemma}\label{Lema-c3.2}
Let $N\gg 1$, $0<\gamma\ll 1$ and $A$ be as defined in \eqref{FT-phi3.2}, $|\xi-N|\leq \gamma$ and
$$(\xi_1, \xi_2, \xi_3)\in R_{1}\times R_{2}\times R_{3}$$ or
$$(\xi_1, \xi_2, \xi_3)\in R_{1}\times R_{3}\times R_{2}.$$ Then one has
\begin{equation}\label{lem-c4.5}
|\xi^3-N^3|\lesssim \gamma N^2, \qquad {\textrm{and}}\qquad |A|\lesssim \gamma N^2.
\end{equation}

\end{lemma}
\begin{proof}
As $|\xi|\sim N$, one has $|\xi^3-N^3|=|\xi-N||\xi^2+\xi N+N^2|\leq 3\gamma N^2$, which implies the first inequality in \eqref{lem-c4.5}. 

Now we move to prove the second inequality in \eqref{lem-c4.5}.  We have 
\begin{equation}\label{x1lem-c4.5}|\xi_1-c_1N|\leq \gamma, \qquad |\xi_2-c_2N|\leq \gamma \qquad{\textrm{and}}\qquad|\xi_3-c_3N|\leq \gamma\end{equation} or
\begin{equation}\label{x2lem-c4.5}|\xi_1-c_1N|\leq \gamma, \qquad |\xi_2-c_3N|\leq \gamma \qquad{\textrm{and}}\qquad|\xi_3-c_2N|\leq \gamma\end{equation}
and consequently $|\xi_j|\sim N$ for all $j=1,2,3$.

Now, using the second condition in \eqref{cond-cs} and \eqref{x1lem-c4.5}
\begin{equation}\label{lem-c4.7}
\begin{split}
|A| &= |\xi_1^3+\alpha\xi_2^3+\alpha\xi_3^3-N^3+N^3-\xi^3|\\
&\leq |\xi_1^3+\alpha\xi_2^3+\alpha\xi_3^3-(c_1N)^3-\alpha(c_2N)^3-\alpha(c_3N)^3|+|N^3-\xi^3|\\
&\leq |\xi_1^3-(c_1N)^3|+\alpha |\xi_2^3-(c_2N)^3|+\alpha |\xi_3^3-(c_3N)^3|+ \gamma N^2\\
&\lesssim \gamma N^2.
\end{split}
\end{equation}
Similarly using \eqref{x2lem-c4.5}, we also have $|A| \lesssim \gamma N^2$.
\end{proof}
 The following Lemma  was proved in  \cite{Oh09}.
\begin{lemma}\label{tadahirolem}
Let $R$ and $\tilde{R}$ be intervals centered at $a$ and $b$ whose dimensions are $2\alpha$. Let $R_0$ be the translate of $R$ centered at $a+b$. Then we have
$$
\chi_R * \chi_{\tilde{R} }(\xi) \geq \alpha \chi_{R_0} (\xi)=\frac12 \,|R|\chi_{R_0}(\xi).
$$
\end{lemma}
\begin{lemma}\label{xmlem}
Let $R_j$, $j=1,2,3$ be intervals centered at $a_j$, $j=1,2,3$ each with dimension $2\alpha $. Let $R_0$ be the translate of $R_1$ centered at $a_1+a_2+a_3$. Then we have
$$
\chi_{R_1} * \chi_{R_2}* \chi_{R_3}(\xi) \geq \alpha^2 \chi_{R_0} (\xi)=\frac{|R_1|^2}{4}\chi_{R_0}(\xi).
$$
\end{lemma}
\begin{proof}
Using the definition of convolution and Lemma \ref{tadahirolem} we obtain
\begin{equation}\label{bil-5}
\begin{split}
\chi_{R_1} * \chi_{R_2}* \chi_{R_3}(\xi)= &\int_{R_1}\chi_{R_2} * \chi_{R_3}(\xi_1-\xi) d\xi_1\,\\
\geq &\int_{R_1} \alpha  \chi_{\tilde{R_0}} (\xi_1-\xi) d\xi_1\\
= & \alpha \chi_{R_1}*\chi_{\tilde{R_0}} (\xi)\\
\geq & \alpha^2 \chi_{R_0} (\xi),
\end{split}
\end{equation}
where $\tilde{R_0}$ be the translate of $R_2$ centered at $a_2+a_3$ and $R_0$ be the  translate of $R_1$ centered at $a_1+a_2+a_3$ and
$(\xi_1, \xi_2, \xi_3)\in R_{1}\times R_{2}\times R_{3}.$
\end{proof}

Now, we retake the proof of Proposition \ref{prop4.1}. We choose $N\gg 1$ and $\gamma =\gamma(N) = \epsilon N^{-2}$, for $\epsilon>0$, so that $|A|\lesssim \epsilon$. With this choice, we have, for any $t>0$
$$\Big|\frac{e^{itA}-1}{A}\Big| =\Big| \frac{\cos tA-1}{A}+i\frac{\sin tA}{A}\Big|\geq \Big|\frac{\sin tA}{A}\Big|\geq \dfrac{t}2.$$
Therefore if $|\xi-N|\leq \gamma$, from \eqref{FT-phi3.2} and Lemma \ref{Lema-c3.1}, we obtain
\begin{equation}\label{FT-phi3.3}
\begin{split}
|\widehat{\Phi_3}(\xi, t)|&\gtrsim |\xi|\gamma^{-\frac32}N^{-3s}\Big|\int_{\R^2}\chi_{R_{1}}(\xi_1)P(\xi_2)P(\xi_3)\frac{e^{itA}-1}{iA}d\xi_1 d\xi_2\Big|\\
&\gtrsim |\xi|\gamma^{-\frac32}N^{-3s} t \Big|\chi_{R_{1}}\convolution\chi_{R_{2}}\convolution\chi_{R_{3}}(\xi)\Big|.
\end{split}
\end{equation}

Now, we move to calculate the $H^s$-norm of $\Phi_3$, using Lemma \ref{xmlem} we obtain
\begin{equation}\label{hs-norm-phi3}
\begin{split}
\|\Phi_3\|_{H^s}&=\|\la\xi\ra^s\widehat{\Phi_3}(\xi,t)\|_{L^2(\R)}\geq \|\la\xi\ra^s\widehat{\Phi_3}(\xi,t)\|_{L^2(|\xi-N|\leq\gamma)}\\
&\gtrsim N^sN\gamma^{-\frac32}N^{-3s} t \Big\|\chi_{R_{1}}\convolution\chi_{R_{2}}\convolution\chi_{R_{3}}(\xi)\|_{L^2(|\xi-N|\leq\gamma)}\\
&\gtrsim N^{1-2s}\gamma^{-\frac32} t \Big\|\gamma^2\chi_{\{|\xi-N|\leq\gamma\}}(\xi)\Big\|_{L^2(|\xi-N|\leq\gamma)}\\
&\gtrsim N^{1-2s}\gamma^{\frac12} t \big|\{|\xi-N|\leq\gamma\}\big|^{\frac12}\\
&\gtrsim N^{1-2s}\gamma t ,
\end{split}
\end{equation}
where in the last step of above inequality used  $\big|\{|\xi-N|\leq\gamma\}\big|\sim \gamma$.  

Therefore, for the choice of $\gamma =\gamma(N) = \epsilon N^{-2}$ we get from \eqref{hs-norm-phi3} that
\begin{equation}\label{hs-norm-phi3.1}
\|\Phi_3\|_{H^s}\gtrsim N^{-1-2s}t\epsilon.
\end{equation}
On the other hand, we have 
\begin{equation}\label{norm phi-psi}
\|(\phi, \psi)\|_{H^s\times H^s} =\|\psi\|_{H^s} \sim 1.
\end{equation}

In view of \eqref{hs-norm-phi3.1} and \eqref{norm phi-psi}, we can conclude that the estimate \eqref{eq4.03} fails to hold when $s<-\frac12$.
\end{proof}


Now, we prove the ill-posedness result stated in Theorem \ref{mainTh-ill1.1}.

\begin{proof}[Proof of Theorem \ref{mainTh-ill1.1}]
For $(\phi, \psi)\in H^s(\R)\times H^s(\R)$, consider the Cauchy problem
\begin{equation}\label{ivp3}
\begin{cases}
\p_tv + \p_x^3v +\p_x(vw^2)=0,\\
\p_tw + \alpha \p_x^3w +\p_x(v^2w) =0,\\
v(x, 0)=\delta\phi,\qquad w(x, 0)=\delta\psi,
 \end{cases}
\end{equation}
where   $\delta >0$ is a parameter. The solution  $(v(x,t), w(x,t)):=(v^{\delta}(x,t), w^{\delta}(x,t))$ of the IVP \eqref{ivp3} depends on the parameter $\delta$. The equivalent integral equation can be written as
\begin{equation}\label{duhamel-c2}
\begin{cases}
v(t,\delta)=\delta U(t)\phi-
\int_0^tU(t-t')\p_x(vw^2)(t')dt'\\
w(t,\delta)=\delta U^{\alpha}(t)\phi-\int_0^tU^{\alpha}(t-t')\p_x(v^2w)(t')dt',
\end{cases}
\end{equation}
where $U(t)$ and $U^{\alpha}(t)$ are the unitary groups describing the solution of the linear part of the IVP \eqref{ivp3}. Note that, by uniqueness of solution we have $v^0(x,t)=0=w^0(x,t)$.

Differentiating \eqref{duhamel-c2} with respect to $\delta$ and evaluating at $\delta=0$ we get
\begin{equation}\label{der-c1}
\begin{cases}
\p_{\delta}v(t,0)=U(t)\phi=:\Phi_1\\ 
\p_{\delta}w(t,0)=U^{\alpha}(t)\psi=:\Psi_1.
\end{cases}
\end{equation}

\begin{equation}\label{der-c2}
\p_{\delta}^2v(t,0)=0 =
\p_{\delta}^2w(t,0).
\end{equation}
\begin{equation}\label{der-c3}
\begin{cases}
\p_{\delta}^3v(t,0)=-6\int_0^tU(t-t')\p_x\big[\Phi_1(\Psi_1)^2\big](t')dt'=:\Phi_3\\ 
\p_{\delta}^3w(t,0)=-6\int_0^tU^{\alpha}(t-t')\p_x\big[(\Phi_1)^2\Psi_1\big](t')dt'=:\Psi_3.
\end{cases}
\end{equation}

If the flow-map is $C^3$ at the origin from $H^s(\R)\times H^s(\R)$ to $C([-T, T];H^s(\R))\times C([-T, T];H^s(\R))$, we must have
\begin{equation}\label{eq-bilin}
\|(\Phi_3, \Psi_3)\|_{L_T^{\infty}H^s(\R)\times H^s(\R)}\lesssim \|(\phi, \psi)\|_{H^s(\R)\times H^s(\R)}^3.
\end{equation}

But from  Proposition \ref{prop4.1}  we have seen that the estimate \eqref{eq-bilin} fails to hold for $s<-\frac12$ if we consider $\phi$ and $\psi$ given respectively by \eqref{phi-hat1} and \eqref{psi-hat1}, and this completes the proof of the theorem.
\end{proof}
\begin{remark}
The ill-posedness result obtained in this work does not depend on the sign of $\alpha$.
\end{remark}

\bigskip
\noindent



\end{document}